\documentclass[11pt,fleqn]{amsartvera}

\usepackage{amsmath}
\usepackage{amssymb}
\usepackage[all]{xy}           
\usepackage{amssymb}     
\usepackage{fullpage}
\usepackage{hyperref}
\usepackage{eucal}
\usepackage{graphicx}
\usepackage{epsfig}
\usepackage[usenames,dvipsnames]{color}
\usepackage{charter}
\usepackage{color}
\numberwithin{equation}{section}
\renewcommand{\labelenumi}{(\roman{enumi})}

\newtheorem{definition}{Definition}[section]
\newtheorem{lemma}[definition]{Lemma}
\newtheorem{theorem}[definition]{Theorem}
\newtheorem{proposition}[definition]{Proposition}
\newtheorem{corollary}[definition]{Corollary}
\newtheorem{remarkth}[definition]{Remark}
\newtheorem{example}[definition]{Example}
\newenvironment{remark}{\begin{remarkth}\upshape}{\hfill$\diamond$\end{remarkth}}

\usepackage{hyperref}
\hypersetup{
    colorlinks,
    citecolor=blue,
    filecolor=blue,
    linkcolor=blue,
    urlcolor=blue
}

\begin{document}

\title{Poisson and near-symplectic structures \\ on generalized wrinkled fibrations in dimension 6}

\author{P. Su\'arez-Serrato}

\address{Instituto de Matem\'aticas - Universidad Nacional Aut\'onoma de M\'exico\\Circuito Exterior, Ciudad Universitaria\\Coyoac\'an, 04510\\Mexico City\\Mexico}

\email{pablo@im.unam.mx}

\author{J. Torres Orozco}

\address{CIMAT, Guanajuato, Mexico }

\email{jonatan@cimat.mx}

\author{R. Vera}

\address{Department of Mathematics - The Pennsylvania State University, State College, PA, 16802, USA}
\email{rvera.math@gmail.com, rxv15@psu.edu}

\begin{abstract}
 We show that generalized broken fibrations in arbitrary dimensions admit rank-2 Poisson structures compatible with the fibration structure. After extending the notion of wrinkled fibration to dimension 6 we prove that these wrinkled fibrations also admit compatible rank-2 Poisson structures. In the cases with indefinite singularities we can provide these wrinkled fibrations in dimension 6 with near-symplectic structures. 
\end{abstract}

\maketitle

\section{Introduction}
Since the seminal work of Donaldson, establishing a correspondence between Lefschetz pencils and symplectic 4-manifolds\cite{D99}, Lefschetz fibrations and its generalizations have played a significant role in symplectic geometry. These are maps to the 2-sphere with a finite number of isolated singular points where the rank of the derivative is zero. In 2005, Auroux, Donaldson, and Katzarkov generalized this approach, introducing what is now known as a {\em broken Lefschetz fibration} or {\em bLf} \cite{ADK05}. There is an additional component in the singularity set of bLfs, a 1-submanifold of indefinite folds. It was shown that a reciprocal geometric structure to bLfs is a near-symplectic form. The latter are closed 2-forms that are non-degenerate outside a collection of circles where they vanish, and they are known to exist on any 4-manifold with $b_{2}^{+} >0$.  Recently, near-symplectic structures and generalized broken Lefschetz fibrations have been studied in higher dimensions \cite{V16}. 
\medskip

From a singularity theory point of view bLfs are not stable. By a stable map it is understood one such that any nearby map in the space of smooth mappings can be perturbed to the original map after a change of coordinates in the domain and codomain. BLfs can be deformed to stable maps. Lekili showed that the unstable Lefschetz singularities of a bLf can be be substituted by cusps, leading to a stable map with only folds and cusps as elements of its critical set \cite{L09}. These mappings are known as {\em wrinkled fibrations}. Furthermore he showed that the near-symplectic structure is preserved under these deformations. 
\medskip 

Poisson geometry has newly entered the picture, particularly in dimension 4. A singular Poisson bivector of rank 2 that vanishes on the singularity set of a bLf and whose symplectic foliation matches the fibres of the map can be given \cite{GSV15}. This in turn implies that on any homotopy class of maps from a 4-manifold to $S^2$ there is such a singular Poisson structure. Similar structures also appear on wrinkled fibrations, where the local models of the Poisson bivectors and induced symplectic forms have been explicitly constructed \cite{ST16}. 
\medskip 

In this work we start by defining a generalization of wrinkled fibrations on dimension 6 based on singularity theory. We then construct Poisson and near-symplectic structures that match the singularities of the fibration and give their local models. Before presenting these constructions, in section \ref{sec:genBLF} we briefly recall the notion of a generalized broken Lefschetz fibration, which serves as a reference for the definition of generalized wrinkled fibrations. Our first observation appears after combining the definition of generalized bLf \cite{V16} together with the results of \cite{Dam89} and \cite{GSV15} on Poisson structures. 

 \begin{theorem}\label{gbLf-poisson}
 Let $M$ and $X$ be closed oriented smooth manifolds with $\dim(M) = 2n, \dim(X) = 2n-2$, and $f\colon M \rightarrow X$ a generalized broken Lefschetz fibration. Then there is a complete singular Poisson structure of rank 2 whose associated bivector vanishes on the singularity set of $f$. If none of its symplectic leaves are, or contain,  $2$--spheres, then this Poisson structure is integrable.
\end{theorem}

The proof of this theorem is a direct application of Theorem \ref{T:Const-Poisson} and the definition of completeness. The integrability condition is verified in the relevant cases, as explained in \ref{integrability}.
\medskip

In section \ref{S:Symp-Struct-Sing} we focus on the Poisson structure on the total space of a generalized wrinkled fibration in dimension 6. We give the general steps for buidling Poisson bivectors around all types of singularities of corank 1, which can be applied on any given dimension. This idea allows us to show the following. 

\begin{theorem}\label{wrinkled-poisson}
Let $M$ be a closed, orientable, smooth $6$--manifold equipped with a generalized wrinkled fibration $f\colon M\rightarrow X$ over a smooth 4-manifold $X$. Then there exists a complete Poisson structure whose symplectic leaves correspond to the fibres of the given fibration structure, and the singularities of both the fibration and the Poisson structures coincide. Moreover, for each singularity, the Poisson bivector and induced symplectic form on the leaves are given by the following equations: 

\begin{tabular}{lclll}
{\em Folds}: & \qquad & Poisson bivectors (\ref{biv:fold}), \eqref{biv:fold-def1}, \eqref{biv:fold-def2} & and & symplectic forms \eqref{FoldForm}, \eqref{FoldFormDefinite1}, \eqref{FoldFormDefinite2}
\nonumber
\\
{\em Cusps}:  & \qquad & Poisson bivector (\ref{biv:cusp}), \eqref{biv:cusp-def1}, \eqref{biv:cusp-def2} & and & symplectic forms (\ref{CuspForm}), \eqref{CuspFormDefinite}
\nonumber
\\
{\em Swallowtail}: & \qquad & Poisson bivector (\ref{biv:swallowtail}), \eqref{biv:swallowtail-def1},   \eqref{biv:swallowtail-def2} & and & symplectic forms (\ref{SwallowtailForm}), \eqref{SwallowtailFormDefinite}
\nonumber
\\
{\em Butterfly}:  & \qquad & Poisson bivector (\ref{biv:butterfly}),  \eqref{biv:butterfly-def1},  \eqref{biv:butterfly-def2} & and & symplectic forms (\ref{ButterflyForm}), \eqref{ButterflyFormDefinite}
\end{tabular}

\noindent If none of its symplectic leaves are, or contain,  $2$--spheres, then this Poisson structure is integrable.
\end{theorem}

The existence of a Poisson structure with the stated properties follows from Theorem \ref{T:Const-Poisson}, previously shown by the first and third named authors together with Garc{\'i}a-Naranjo \cite{GSV15}. The proof of this theorem follows from an application of Theorem \ref{T:Const-Poisson} and the definition of completeness.\medskip 

These results allow us to present in section \ref{morita-examples} countably many examples of Poisson structures on the same underlying smooth manifold that are Morita inequivalent. In our examples the leaves of the symplectic foliations change topology, as the fibrations involved undergo deformations. 
\medskip 

In section \ref{S:Symp-Struct-Sing}, the local models for the bivectors are shown to hold true. Then in section \ref{S:symp-forms} we prove that the local models for the symplectic forms on the leaves are also the claimed ones.
\medskip 

Finally, in section \ref{sec:near-sympletic} we turn to near-symplectic geometry. Explicit models of near-symplectic forms have appeared in previous work \cite{Ho204, ADK05, L09, V16}. Here we show a further construction of local near-symplectic forms that follows a general scheme for all the singularities of a generalized wrinkled fibration. Assuming the global conditions for the existence of a near-symplectic structure on a $2n$--manifold with a generalized bLf are met and constructing the local forms around the new singularities we obtain:

\begin{theorem}\label{thm:near-symplectic}
Let $M$ be a closed oriented 6-manifold, $(X, \omega_X)$ a closed symplectic 4-manifold, and $f\colon M \rightarrow X$ a generalized wrinkled fibration. Denote by $Z$ the singularity set of $f$, a 3-submanifold of $M$. Assume that there is a class $\alpha \in H^2(M)$, such that it pairs positively with every component of every fibre, and $\alpha|_{Z} = [\omega_X|_{Z}]$.  Then there exist a near-symplectic form $\omega$ on $M$ with singular locus $Z$ such that it restricts to a symplectic form on the smooth fibres of the fibration.
\end{theorem}

The proof of this theorem appears in section \ref{sec:proof-near-symplectic}.  The global construction of the near-symplectic structure on the total space follows as in dimension 4 \cite{ADK05, L09} with the modifications for higher dimensions introduced by the third named author \cite{V16}.  The most substantial difference concerns the construction of the local forms around the new singularities.  

Notice that our Theorem \ref{wrinkled-poisson} presents extensive possibilities for Poisson structures supported on wrinkled fibrations. Whereas our Theorem \ref{thm:near-symplectic} is limited in the singularities used by the restrictions imposed by near-symplectic structures. In particular, there will be no analogue of the results of Auroux-Donaldson-Katzarkov (for bLf's) or of Lekili (for wrinkled fibrations) in dimension 4, which give general conditions for all these fibrations to admit compatible near-symplectic structures. 

 Here we examine near-symplectic forms and singular Poisson bivectors in relation to extensions of Lefschetz fibrations using maps coming from singularity theory. The recent work of Cavalcanti and Klaasse \cite{CK16} also considered Lefschetz fibrations in singular symplectic and Poisson geometry. In contrast to our work, Cavalcanti and Klaasse have connected achiral Lefschetz fibrations to log-symplectic and folded symplectic structures. An even dimensional manifold $M$ is said to be {\em log-symplectic} if it is equipped with a Poisson bivector $\pi$ such that the Pfaffian $\pi^n$ is transverse to the zero section in $\Lambda^{2n} TM$, $M$ is called {\em folded symplectic} if it comes with a closed 2-form $\omega$ such that $\omega^n$ vanishes transversally in $\Lambda^{2n}T^{*}M$.  The geometric structures studied in this work differ from the ones considered by Cavalcanti and Klaasse. It is not yet clear what is the precise relationship between broken Lefschetz fibrations and log-symplectic manifolds.
\\
\\
{\it Acknowledgements:} We warmly thank Yank{\i}  Lekili for answering detailed questions about his paper. His explanations allowed us to complete our computations for the near-symplectic forms. We also thank Alan Weinstein for commenting on this paper. PSS thanks PAPIIT-UNAM and CONACyT-M{\'e}xico for supporting research activities, to the organizers of the meeting {\it 'Gone fishing 2016'} in Boulder, and IPAM in UCLA where some of this work was done.

\section{Preliminaries}

\subsection{Wrinkled and Broken Lefschetz Fibrations}\label{sec:genBLF}
We start by recalling the definition of a broken Lefschetz fibration. Regular fibres are 2-dimensional, smooth and convex and singular fibres present an isolated nodal singularity.  

\begin{definition}
\label{D:BLF}
On a smooth, closed 4-manifold $X$, a {\em broken Lefschetz fibration} or {\em BLF} is a smooth map $f\colon X \rightarrow S^2$ that is a submersion outside the singularity set.  Moreover, the allowed singularities are of the following type:
\begin{enumerate}
\item {\em indefinite fold singularities}, also called {\em broken}, contained in the smooth embedded 1-dimensional submanifold $\Gamma \subset X \setminus  C$, which are locally modelled by the real charts
$$ \mathbb{R}^{4} \rightarrow \mathbb{R}^{2} ,  \quad \quad  (t,x_1,x_2,x_3) \mapsto (t, - x_{1}^{2} + x_{2}^{2} + x_{3}^{2}),$$
\item {\em Lefschetz singularities}:  finitely many points  $$C= \{ p_1, \dots , p_k\} \subset X,$$  which are locally modeled by complex charts
$$ \mathbb{C}^{2} \rightarrow \mathbb{C}  ,  \quad \quad (z_1, z_2) \mapsto z_{1}^{2} + z_{2}^{2}.$$
\end{enumerate}
\end{definition}

We recollect a few concepts of singularity theory before defining a generalized bLf.  Let $f\colon M^n \rightarrow X^q$ be a smooth map between two smooth manifolds with $\dim(M) \geq \dim(X)$ and differential map $df\colon TM \rightarrow TX$. A point $p\in M$ is {\em regular} if the rank of $df_{p}$ is maximal. In this case $f$ is a submersion at $p$. If $\textnormal{Rank}(df_p) < \dim(X)$, then a point $p\in M$ is called a {\em singularity of $f$}. Let $k = \dim(X) - \textnormal{Rank}(df_p)$ denote the corank of $f$. The set $\Sigma_k = \lbrace p\in M \mid \textnormal{corank}(df_p) = k \geq 1 \rbrace$ is known as the {\it singularity set} or singular locus of $f$.  For generic maps, $\Sigma_k$ are submanifolds of $M$. As we can see from the definition, there can be different singularity sets depending on the corank of $f$. In this work we will focus on singularities of corank 1. The elements of the set $\Sigma_{1}$ satisfying $T_p \Sigma_f + \ker(f) = T_p M$ are called {\it fold} singularities of $f$.  A mapping $f\colon M \rightarrow X$ is then known as a {\it submersion with folds}, if it is a submersion outside the set of fold singularities.  In particular, a submersion with folds restricts to an immersion on its fold locus (see Lemma 4.3 p.87 \cite{GG}).  Submersions with folds are related to stable maps. By a stable $f$ we mean that any nearby map $\tilde{f} \in C^{\infty}(M,X)$ is equivalent to $f$ after a smooth change of coordinates in the domain and range. 

Folds are locally modelled by real coordinate charts $\mathbb{R}^n \rightarrow \mathbb{R}^q$ with $n>q$ and coordinates
$$\left( x_1, \dots , x_n \right) \rightarrow \left( x_1, \dots, x_{q-1}, \pm x^2_{q} \pm x^2_{q+1} \pm \dots \pm x^2_{n}\right) $$
As we can see from the above parametrization, when $q=1$, submersions with folds correspond precisely to Morse functions on $M$.  It is well known that Morse functions are dense in the set of smooth mappings from any $n$-dimensional manifold $M$ to $\mathbb{R}$. There is an equivalent statement for maps with a 2-dimensional target space.  Assumming that $f$ is generic then $\Sigma_1$ is a submanifold, and the restriction of $f$ at $\Sigma_1$ gives a smooth map between manifolds that can also have generic singularities. When the target map is of dimension 2, there is one extra type of generic singularity called cusp. {\it Cusps} are points $p\in \Sigma_1$ such that $T_p \Sigma_1(f) = \ker(df_p)$, and they are parametrized by real charts $\mathbb{R}^n \rightarrow \mathbb{R}^2$  with coordinates
$$\left( x_1, \dots , x_n \right) \rightarrow \left( x_1, x_2^3 + x_1\cdot x_2 \pm  x^2_{3} \pm \dots \pm x^2_{n}\right) $$

\noindent Folds and cusps are the singularities of a wrinkled fibration. 

\begin{definition}\label{D:wrinkled}
A purely wrinkled fibration is a submersion $f$ on a closed 4--manifold $X$ to a closed surface having two types of singularities:
\begin{enumerate}
\item {\em indefinite fold singularities} contained in the smooth embedded 1-dimensional submanifold $\Gamma \subset X$, which are locally modelled by the real charts
$$ \mathbb{R}^{4} \rightarrow \mathbb{R}^{2} ,  \quad \quad  (t,x_1,x_2,x_3) \mapsto (t, - x_{1}^{2} + x_{2}^{2} + x_{3}^{2}),$$

\item {\em cusps}, finitely many points contained in the set  $B= \{ p_1, \dots , p_k\} \subset X,$  which are locally modeled by real charts
$$ \mathbb{R}^{4} \rightarrow \mathbb{R}^2  ,  \quad \quad (t,x_1,x_2,x_3) \mapsto (t,  x_{1}^{3} + 3 t x_1  + x_{2}^{2} - x_{3}^{2}).$$
\end{enumerate}
\end{definition}

\noindent The following theorem shows that generic maps from any $n$--dimensional manifold to a 2--dimensional base have folds and cusps. 
\begin{theorem}\cite{GG}
A generic smooth map $M^n \rightarrow N^2$ has folds and cusps singularities. 
\end{theorem}

In this context, wrinkled fibrations are generic maps defined on a smooth 4-manifold with image on the 2-sphere, and broken Lefschetz fibrations are submersions with folds and Lefschetz singularities. The latter are natural in the symplectic setting. As it was shown by Donaldson and Gompf, there is a correspondence between symplectic 4-manifolds and Lefschetz fibrations. Yet, Lefschetz singularities are not stable from the point of view of singularity theory.  Lekili showed that Lefschetz singularities can be transformed into cusps yielding to a wrinkled fibration. As a consequence, we can modify a bLf into a submersion with folds and cusps, which are stable and dense.

We proceed now to higher dimensions. To start, consider the definition of broken Lefschetz fibrations in higher dimensions \cite{V16}.

\begin{definition}\label{genBLF}
Let $M,X$ be smooth manifolds of dimensions $2n$ and $2n-2$. By a {\em generalized broken Lefschetz fibration} we mean a submersion $f\colon M \rightarrow X$ with two types of singularities:
\\
\\
1. {\em Indefinite fold singularities}, locally modeled by:
\begin{align}
\mathbb{R}^{2n} &\rightarrow \mathbb{R}^{2n-2}
\nonumber
\\
(t_1, \dots, t_{2n-3}, x_1,x_2,x_3)& \mapsto (t_1, \dots, t_{2n-3},  -x_1^2 + x_2^2 + x_3^2)
\nonumber
\end{align}
The fold locus is an embedded codimension 3 submanifold. We denote it by $Z$.  Singular fibres have again at most one singularity on each fibre,  but this time crossing $Z$ changes the genus of the regular fibre by one. Throughout this work we assume that the singular fibres do not intersect each other. 
\\
\\
2. {\em Lefschetz-type singularities}, locally modelled by:
\begin{align}
\mathbb{C}^n &\rightarrow \mathbb{C}^{n-1}
\nonumber
\\
(z_1, \dots , z_n) &\rightarrow (z_1, \dots z_{n-2}, z_{n-1}^2 + z_{n}^2)
\nonumber
\end{align}

These singularities are contained in codimension 4 submanifolds cross a Lefschetz singular point. We denote the set of Lefschetz-type singularities by $C$. Each singular fibre presents at most one singularity on each fibre.  On a piece of the fibre, this can be depicted as a local cone that collapses at the origin where $z_{n-1}^2 + z_{n}^2 = 0$.  Nearby fibres are smooth.  In the local description on a piece of a fibre, the cone opens up again and it is convex. 
\end{definition}

\subsection{Generalized Wrinkled Fibrations}
Stable maps of $M^n \rightarrow X^q$ are dense in $C^{\infty}(M,X)$ if and only if the pair $(n,q)$ satisfies certain conditions depending on the dimension $q$ of the target manifold $X$ and the difference $(n-q)$. We refer the reader to \cite{GG, M71} for a detailed account. In particular, in the case of $M^6 \rightarrow X^4$ we have the following characterization. 

\begin{theorem}\cite{M71, GG}
A generic smooth map $M^6 \rightarrow N^4$ has folds, cusps, swallowtails, and butterflies singularities. 
\end{theorem}

\noindent This suggests the following definition. 

\begin{definition}\label{genWLF-dim-6}
On a smooth 6-manifold $M$ a {\em generalized wrinkled fibration} $f\colon M�\rightarrow X$ is a submersion to a smooth closed 4-manifold $X$ with the following four indefinite singularities each locally modelled by real charts $\mathbb{R}^6 \rightarrow \mathbb{R}^4$

1. {\em folds}
$$ (t_1, t_2, t_3, x_1,x_2,x_3) \mapsto (t_1, t_2, t_3,  -x_1^2 + x_2^2 + x_3^2)$$

2. {\em cusps} 
$$ (t_1, t_2, t_3, x_1,x_2,x_3) \mapsto (t_1, t_2, t_3,  x_1^3  - 3 t_1 \cdot x_1 + x_2^2 -  x_3^2) $$

3. {\em swallowtails} 
$$(t_1, t_2, t_3, x_1, x_2, x_3) \mapsto ( t_1, t_2, t_3,    x_1^4 + t_1 x_1^2 + t_2 x_1 + x_2^2  -  x_3^2)$$

4. {\em butterflies} 

$$(t_1, t_2, t_3, x_1, x_2, x_3) \mapsto ( t_1, t_2, t_3,    x_1^5 + t_1 x_1^3 + t_2 x_1^2 + t_3 x_1  + x_2^2 - x_3^2)$$
\end{definition}

\subsection{Poisson Manifolds}\label{ss:Poisson} This section contains basic facts about Poisson geometry that we will use throughout the paper. We refer the interested readers to \cite{V94, DZ05, LGPV13} for further details.

\subsubsection{Poisson Structures}
\begin{definition}
\label{D:Poisson}
 A {\em Poisson bracket} (or a {\em Poisson structure}) on a smooth manifold $M$ is a bilinear operation
 $\{\cdot , \cdot \}$ on the set $C^\infty(M)$ of real valued smooth functions on $M$ that satisfies
 \begin{enumerate}
\item[(i)] $( C^\infty(M) , \{\cdot , \cdot \})$ is a Lie algebra.
\item[(ii)]  $\{gh, k\}=g\{h, k\}+ h\{g, k\}$ for any $g,h,k\in  C^\infty(M)$.
\end{enumerate}
\end{definition}
A manifold $M$ with such a Poisson bracket is called a {\em Poisson manifold}. If $(M, \omega)$ is a symplectic manifold, we may use the symplectic form $\omega$ to produce a Poisson structure. The bracket of $M$ is defined by
\begin{equation*}
\{g,h\}=\omega(X_g,X_h).
\end{equation*}
Hamiltonian vector fields $X_h$ are defined by ${\bf i}_{X_h}\omega =dh$.

Thus, using property (ii) from Definition \ref{D:Poisson} we may define Hamiltonian vector fields for Poisson manifolds.  Given a function $h\in C^\infty(M)$  we assign it the {\em Hamiltonian vector field} $X_h$, defined via
\begin{equation*}
X_h(\cdot )=\{\cdot , h \}.
\end{equation*}

It follows from  (ii) that a Poisson bracket $\{g,h\}$ depends only on the first derivatives of $g$ and $h$.
Hence the Poisson bracket may be considered as defining a bivector field $\pi$ defined by
\begin{equation}\label{eq:bracket-bivector}
\{g,h\}=\pi(dg,dh).
\end{equation}

\noindent Let $\pi$ be a Poisson bivector, for coordinates $(x^1, \dots , x^n)$, we give a local expression
\begin{equation*}
\pi(x)=\frac{1}{2}\sum_{i,j=1}^n\pi^{ij}(x)\frac{\partial}{\partial x^i}\wedge \frac{\partial}{\partial x^j}.
\end{equation*}
Here $\pi^{ij}(x)=\{x^i,x^j\}=-\{x^j,x^i\}$.
\newline

A Poisson bracket satisfies the Jacobi identity, this is expressed by a partial differential equation  for the components of the Poisson bivector \cite{LGPV13}.  Maps that preserve Poisson brackets are called Poisson maps. There are also maps which are called {\em anti-Poisson}:

\begin{definition}[anti-Poisson maps]
A map $\varphi : P\to P'$ between two Poisson manifolds $P$ and $P'$ is called an {\em anti-Poisson} map if
\begin{equation*}
\{ f_1 \circ \varphi ,  f_2 \circ \varphi \}_{P'} = - \{ f_1 ,  f_2 \} \circ \varphi 
\end{equation*}
holds for any two smooth functions $f_1$ and $f_2$ on $P$.
\end{definition}

Let $\pi$ be a bivector on $M$, $q\in M$, and  $\alpha_q\in T_q^*M$. It is possible to define a bundle map:
$$\mathcal{B}\colon T^*M\to TM\,;\, \mathcal{B}_q(\alpha_q)(\cdot)=\pi_q(\cdot ,\alpha_q) $$

\noindent If $\pi$ is a Poisson bivector, we obtain that $X_h=\mathcal{B}(dh)$. We then define the {\em rank} of  $\pi$ at  $q\in M$ to be equal to the rank of $\mathcal{B}_q\colon T^*_qM\to T_qM$. This coincides with the rank of the matrix $\pi^{ij}(x)$.

The distribution defined by $\mathcal{B}_q$  on $ T_qM$ is called the {\em characteristic distribution} of $\pi$. We know by the {\em Symplectic Stratification Theorem} that the characteristic distribution of a Poisson bivector $\pi$ gives rise to a (possibly singular) foliation by symplectic leaves. This foliation is integrable in the sense of Stefan-Sussman (see Theorem 2.6 in \cite{V94}).

Call  $\Sigma_q$  the symplectic leaf of $M$ through the point $q$. As a set $\Sigma_q$ can be also considered as the collection of points that may be joined via piecewise smooth integral curves of Hamiltonian vector fields. Note that if $\omega_{\Sigma_q}$ is the symplectic form on  $\Sigma_q$, $T_q\Sigma_q$ is exactly the characteristic distribution of $\pi$ through $p$. Therefore, given $u_q, v_q \in T_q\Sigma_q$
there exist $\alpha_q, \beta_q\in T^*_qM$ such that it image $\mathcal{B}_q$ is $u_q$ and $v_q$, respectively. Using this we can describe the symplectic form $\omega_{\Sigma_q}$:
\begin{equation}
\label{E:Symp-form-gen}
\omega_{\Sigma_q}(q)(u_q, v_q)=\pi_q(\alpha_q, \beta_q)=\langle \alpha_q,  v_q \rangle=-\langle \beta_q,  u_q \rangle.
\end{equation}

The rank varies as the dimensions of the symplectic leaves do.  When the rank of the characteristic distribution of a bivector is less than or equal to two the following holds \cite{GSV15}:
\begin{proposition}
\label{Prop:rank2}
\begin{enumerate}
\item[]
\item If $\pi$ is a bivector field on $M$ whose characteristic distribution is integrable and has rank less than or equal to two
at each point, then $\pi$ is Poisson.
\item Let $\pi$ be a Poisson structure on $M$ whose rank at each point is less than or equal to two. Then $\pi_1:=k\pi$ is also a Poisson
structure where $k\in C^\infty(M)$ is an arbitrary non-vanishing function.
\end{enumerate}
\end{proposition}

\noindent We will describe the bivectors locally using certain Casimir functions.

\begin{definition}
\label{D:Casimir}
Let $M$ be a Poisson manifold. A function $h\in C^\infty(M)$ is called a {\em Casimir} if $\{h,g\}=0$ for every $g\in C^\infty(M)$. Equivalently $\mathcal{B}(dh)=0$.
\end{definition}

\begin{theorem}[\cite{GSV15}]
\label{T:Const-Poisson}
Let $M$ be an orientable $n$-manifold, $N$ an orientable $n-2$ manifold, and $f:M\to N$ a smooth map.
Let $\mu$ and $\Omega$ be orientations of $M$ and $N$ respectively. The bracket on $M$ defined by
\begin{equation}
\label{E:Def-Intrinsic}
\{g,h\}\mu=k\,dg\wedge dh \wedge f^*\Omega
\end{equation}
where $k$ is any non-vanishing function on $M$ is Poisson. Moreover, its symplectic leaves are
\begin{enumerate}
\item[(i)] the 2-dimensional leaves $f^{-1}(s)$ where $s\in N$ is
a regular value of $f$, \item[(ii)] the 2-dimensional leaves
$f^{-1}(s)\setminus \{\mbox{Critical Points of $f$}\}$ where $s\in
N$ is a singular value of $f$. \item[(iii)] the 0-dimensional
leaves corresponding to each critical point.
\end{enumerate}
\end{theorem}

Formula \eqref{E:Def-Intrinsic} appeared in \cite{Dam89}
(attributed to H. Flaschka and T. Ratiu).

\subsubsection{Completeness}
We recall the concept of a complete Poisson manifold
\begin{definition}
A Poisson manifold $M$ is said to be complete if every Hamiltonian
vector field on $M$ is complete.
\end{definition}
\noindent Notice that $M$ is complete if and only if every symplectic leaf is bounded in the
sense that its closure is compact. We can also talk about complete Poisson maps. 

\begin{definition}
A Poisson map $J: Q \to P$ between the Poisson manifolds $Q$ and $P$ is {\em complete} if the hamiltonian vector field $X_{J^{\ast} f}$ is complete whenever $X_{f}$ is complete, $f$ in $C^{\infty}(P)$.
\end{definition}

\subsection{Integrability}\label{integrability}

\subsubsection{Lie Algebroids}

A {\bf Lie algebroid} $(A,B,\rho, [\, , \, ] )$ is a vector bundle $A$ over a manifold $B$ together with a bundle map $\rho : A \to TB$, called the {\bf anchor}, and a Lie bracket $[\, , \, ]$ on the real vector space $\Gamma(A)$ of sections of $A$ such that for every $X,Y$ in $\Gamma(A)$ and smooth real function $f$ of $B$:
$$[X,fY]=f[X,Y]+L_{\rho_{\ast}X} fX$$

Here $\rho_{\ast}$ is the induced map on sections, and $L$ the Lie derivative. 

\subsubsection{Lie Groupoid}

A {\bf groupoid} is small category where all the morphisms are invertible. This consists of a set of morphisms $G$ and a set of objects $B$. There exist surjective maps $s,t:G\to B$ called the {\bf source} and {\bf target} maps, respectively. Let 
$$G^{(2)}:= \{ (u,v)\in G\times G \, : \, s(u)=t(v) \}.$$ 
There exists a {\bf multiplication} map;
$$ G^{(2)}\to G \, ; \,(u,v)\mapsto uv  $$
an {\bf inverse} map;
$$  G\to G\,;\, u,v)\mapsto u\mapsto (u^{-1})$$
and an {\bf identity} bisection:
$$\epsilon : B \to G $$

These comply with the following axioms, for every $u,v,w$ in $G$, and $b$ in $B$. 

$\begin{array}{ccccccccc}
s(uv)&=&s(v),& t(uv)&=&t(u), & \epsilon(b)v&=&v, \\
u\epsilon(b)&=&u, & (uv)w&=&u(vw), & s(u^{-1})&=&t(u), \\
t(u^{-1})&=&s(u), & uu^{-1}&=&\epsilon(t(u)), & uu^{-1}&=&\epsilon(s(u)).
\end{array}$
 
The vector bundle ${\rm ker} (ds)|_{\epsilon(B)}$ has a natural structure of a Lie algebroid ove $B$ with anchor $dt$, its Lie bracket is induced by the multiplication.

Denote by ${\rm Lie}(G)$ the Lie algebroid of the Lie groupoid $G$. Notice that not every Lie algebroid occurs as ${\rm Lie}(G)$ for some $G$. Those which do are called {\bf integrable}.

\subsubsection{Algebroids on cotangent bundles to Poisson manifolds}

The cotangent bundle of a Poisson manifold $M$ can be given a Lie algebroid structure. The Poisson bivector field induces an anchor map, for every $x$ in $M$ and $\sigma$ in $T^{\ast}_{x}M$: 
$$\pi^{\#}(x)(\sigma)=\pi(x) (\sigma, \cdot)$$

The Lie bracket in this structure was introduced by Koszul for $1$--forms:
$$[df, dg]:=d \{ f , g \}$$

When this Lie algebroid is integrable, its associated Poisson manifold is said to be integrable. 
 
Crainic and Fernandes found general obstructions for the integrability of Lie
algebroids. They proved that Poisson manifolds whose symplectic leaves have trivial second homotopy groups are integrable \cite{CF03}.

\begin{definition}
A {\em symplectic realization} of a Poisson manifold $P$ is a Poisson map from a symplectic manifold to $P$.
\end{definition}

\subsection{Morita equivalences}

We will now briefly recall the notion of Morita equivalence for integrable Poisson manifolds (see \cite{BW05, X91, X92}).

\begin{definition}[complete full dual pair]
Let $S$ be a symplectic manifold. A pair of Poisson maps $P_{1} \leftarrow S \to P_{2}$ is called a {\em dual pair} if the $J_1$- and $J_2$- fibres are the symplectic orthogonal of each other. Such a pair is called {\em full} if $J_1$ and $J_2$ are surjective submersions. If both $J_1$ and $J_2$ are complete, it is called {\em complete}.
\end{definition}

\begin{definition}[Morita equivalence for integrable Poisson manifolds]
A pair of integrable Poisson manifolds $P_1$ and $P_2$ are called {\em Morita equivalent} if there exists a symplectic manifold $S$ with a complete Poisson map $J_1 : S\to P_1$ and a complete anti-Poisson map  $J_2 : S\to P_2$ so that $P_{1} \leftarrow S \to \overline{P_{2}}$ is a complete full dual pair for which the $J_1$- and $J_2$- fibres are simply connected. 
\end{definition}

\noindent For the readers' convenience, we include the next well known statement:

\begin{lemma}[Morita equivalences and fundamental groups of leaves]\label{morita-ineq}
For Morita equivalent Poisson manifolds, corresponding symplectic leaves have isomorphic fundamental groups.
\end{lemma}

\begin{proof}
Let $P_1$ and $P_2$ be two Morita equivalent Poisson manifolds and $P_{1} \leftarrow S \to \overline{P_{2}}$ be the associated complete full dual pair. Consider symplectic leaves $L_1$ in $P_1$ and $L_2$ in $P_2$ such that $N=J_1^{-1}(L_1)=J_2^{-1}(L_2)$. Then $J_{i}$ maps $N$ onto $L_{i}$, and the maps have simply connected fibres. So there exists an induced isomorphism of fundamental groups; $\pi_{1}(L_1)\cong \pi_{1}(N) \cong \pi_{1}(L_2)$.\end{proof}

\subsubsection{Morita inequivalent structures}\label{morita-examples}

\begin{example} The near-symplectic cobordisms described by Perutz \cite{P07} for 4-dimensional manifolds can be used to describe examples of Morita inequivalent Poisson structures. For example, suppose there is a near-symplectic cobordism where the number of connected components of the critical set in the base changes, then so does the topology of the fibres in the respective fibrations in the start and end of the cobordism. Assume that none of the fibres in these fibrations were or contained 2-spheres in the initial part of the cobordism, and that the genus of the fibration increases along the cobordism. Then, by lemma \ref{morita-ineq}, the associated Poisson structures on the boundaries of the cobordism are not Morita equivalent.
\end{example}

\begin{example}
The deformations of wrinkled fibrations introduced by Lekili \cite{L09} have been used to describe Poisson structures on the associated fibrations \cite{ST16}. In a similar way to the previous example assume that $M_0$ is a closed smooth oriented 4-manifold with a wrinkled fibration whose fibres do not contain 2-spheres. Then the associated Poisson structure $\Pi_0$ is integrable \cite{ST16}. Perform one of Lekili's deformations on $(M_0,\Pi_0)$ which increases the fibre genus, then the resulting manifold $(M_1, \Pi_1)$ is Poisson \cite{ST16}. Then Lemma \ref{morita-ineq} implies $(M_0,\Pi_0)$ and $(M_1, \Pi_1)$ are not Morita equivalent. Iterating this process exhibits a countable abundance of Morita inequivalent structures on the same underlying smooth $4$-manifold. 
\end{example}

\section{Local Poisson bivectors for the proof of Theorem \ref{wrinkled-poisson} .}\label{S:Symp-Struct-Sing}

We will now give explicit local descriptions for the Poisson structures and the corresponding symplectic forms in a neighbourhood singularities of  generalized wrinkled fibrations in dimension 6. All of the expressions that we will give depend abstractly on an arbitrary choice of a non-vanishing function $k$ in $C^\infty(M)$. See Proposition \ref{Prop:rank2}. Before proceeding we will describe the general strategy employed to find the local bivectors.
\medskip

\noindent {\bf Step 1:} Consider the coordinate functions $C_1, C_2, C_3, C_4$ that describe each fibration as Casimir functions for the Poisson structure that we want to find.
\medskip

\noindent {\bf Step 2:} Calculate the differentials $dC_i$, $i=1, 2, 3, 4$.
\medskip

\noindent {\bf Step 3:} We use formula \ref{E:Def-Intrinsic} to compute the
skew-symmetric matrix with entries:
\begin{equation*}
\pi^{ij}=\{x^i,x^j\}\mu=\,dx^i\wedge dx^j \wedge dC_1\wedge dC_2\wedge dC_3\wedge dC_4.
\end{equation*}

This matrix will then annihilate $dC_i$, $i=1, 2, 3, 4$. It will give the endomorphism $\mathcal{B}$ associated to a Poisson structure with $dC_i$, $i=1, 2, 3, 4$, as Casimirs.
The components of the bivector field will be given by:
\begin{equation*}
\{x^i, x^j\}=\det\left( \epsilon^i, \epsilon^j, dC_1, dC_2 , dC_3, dC_4\right)
\end{equation*}
Here $\epsilon^i$ is the $6\times1$ canonical basis column vector, whose $i$-th
component is $1$ and all others are zero.

\medskip

\noindent {\bf Step 4:} We then write the Poisson bivector using the skew-symmetric matrix entries.

\subsection{General criterion for constructing Poisson bivectors on singularities}
We extend the previous strategy to manifolds of dimension $2n$ when we have a singular submersion with singularities of corank 1.  The following construction will describe a procedure that can be used to compute local expressions of Poisson structures and their corresponding symplectic forms. We will implement this scheme to study the 6-dimensional case. An explicit computation of the local models in dimension 6 appears in appendix \ref{Appendix:Poisson}.

\begin{proposition}\label{bivector-criterion}
Let $q$ be a point that either has complex coordinates $q=(z_1, z_2, \dots z_n)$ or real coordinates $(t_1, t_2, \dots, t_{2n-4}, t_{2n-3}, x_1, x_2, x_3)$. Let $f$ be a smooth map given as either $f\colon \mathbb{C}^n \to \mathbb{C}^{n-1}$ or $f\colon \mathbb{R}^{2n}\to \mathbb{R}^{2n-2}$ such that 
$$f(q)=(z_1, \dots, z_{n-2}, f_o(z_{n-1}, z_{n}))$$ 
or 
$$f(q)=\left(t_1, t_2, \dots, t_{2n-4}, t_{2n-3}, f_o(t_1, \dots, t_{2n-3}, x_1, x_2, x_3)\right),$$ respectively. Here $f_o$ is a smooth map which depends only on the last coordinates $z_{n-1}, z_n$ or $x_1, x_2, x_3$. Then we can produce a Poisson structure associated to the local model given by  $f$. The Poisson bivector has the form:
\begin{equation*}
\pi= \left(\begin{array}{cccccccc}
0 &\cdots&0 & 0 & 0 & 0& 0\\
\vdots &\ddots &\vdots & \vdots & \vdots & \vdots& \vdots\\
0 & \cdots& 0 & 0& 0 & 0 & 0\\ 
0 &\cdots& 0&\pi^{11}&\pi^{12}&\pi^{13} & \pi^{14}\\ 
0 &\cdots & 0&\pi^{21}&\pi^{22}&\pi^{23}  & \pi^{24}\\    
0 &\cdots &0&\pi^{31}&\pi^{32}&\pi^{33} &\pi^{34} \\
0 &\cdots& 0&\pi^{41}&\pi^{42}&\pi^{43} & \pi^{44}
\end{array} 
\right)
\end{equation*}
 where $\pi^{ij}$ is the Poisson bivector of the map $f_o$. Then $\pi^{ii}=0$ and $\pi^{ij}=\pi^{ji}$. Therefore the Poisson bivector has the local form:
 \begin{equation*}
 \pi(x)=\sum_{i, j=1}^4\left[\pi^{ij}\frac{\partial }{\partial x^i}\wedge \frac{\partial }{\partial x^j}\right]
\end{equation*}

\end{proposition}

\begin{proof}
In the case when $f$ is a complex map we use the real and imaginary parts of each coordinate function as a Casimir function for the Poisson structure that we want to find. That is, we will have $2n-2$ Casimir functions: 
\begin{eqnarray*}
C_i &=&Re(z_i) \quad 1\leq i \leq n-2\\
C_{i+n-2} &= &Im(z_i)\quad 1\leq i \leq n-2\\
C_{2n-3} &= &Re (f_o(z_{n-1}, z_n))\\
C_{2n-2} &= &Im (f_o(z_{n-1}, z_n))
\end{eqnarray*}

Now we compute the differential matrix of the map. It gives a matrix with a $2\times 4$-block corresponding to the derivatives of the real and complex part of $f_o$ and ones on the principal diagonal.
\begin{equation}\label{array:differential-Casimir-N-dimC}
D=\left( \begin{array}{cccccc}
1 &\dots & 0 & 0& 0\\
\vdots &\ddots & \vdots & \vdots& \vdots\\
0 & \cdots & 1 & 0 & 0 \\ 
0 &\cdots &0& \frac{\partial C_{2n-3}}{\partial t_{2n-3}}&\frac{\partial C_{2n-2}}{\partial t_{2n-3}}\\ 
0 &\cdots &0 &\frac{\partial C_{2n-3}}{\partial x_1}&\frac{\partial C_{2n-2}}{\partial x_1}  \\    
0 &\cdots &0&\frac{\partial C_{2n-3}}{\partial x_2}&\frac{\partial C_{2n-2}}{\partial x_2}\\
0 &\cdots &0& \frac{\partial C_{2n-3}}{\partial x_3}&\frac{\partial C_{2n-2}}{\partial x_3}
\end{array} 
\right)
\end{equation}

According to the formula (\ref{E:Def-Intrinsic}) the coefficients of the bivector matrix are given by
\begin{equation*}
\pi^{ij}=Det\left[
\left( \begin{array}{cccccccc}
1 &0 & 0 & 0& 0& \epsilon^1_i &\epsilon^{1}_j\\
\vdots &\ddots & \vdots & \vdots& \vdots& \vdots& \vdots\\
0 & \cdots & 1 & 0 & 0 & \epsilon^{2n-4}_i &\epsilon^{2n-4}_j\\ 
0 &\cdots &0& \frac{\partial C_{2n-3}}{\partial t_{2n-3}}&\frac{\partial C_{2n-2}}{\partial t_{2n-3}}& \epsilon^{2n-3}_i &\epsilon^{2n-3}_j\\ 
0 &\cdots &0 &\frac{\partial C_{2n-3}}{\partial x_1}&\frac{\partial C_{2n-2}}{\partial x_1}  & \epsilon^{2n-2}_i &\epsilon^{2n-2}_j\\    
0 &\cdots &0&\frac{\partial C_{2n-3}}{\partial x_2}&\frac{\partial C_{2n-2}}{\partial x_2}& \epsilon^{2n-1}_i &\epsilon^{2n-1}_j\\
0 &\cdots &0& \frac{\partial C_{2n-3}}{\partial x_3}&\frac{\partial C_{2n-2}}{\partial x_3}&\epsilon^{2n}_i &\epsilon^{2n}_j
\end{array} 
\right)\right]
\end{equation*}

where $\epsilon^{k}_i$ and $\epsilon^{k}_j$ are canonical basis column vectors, whose $i-$th  and $j-$th component, respectively is $1$ and all others are zero. Note that it contains a identity matrix of dimension $(2n-4)\times (2n-4)$. Therefore the determinant is the same as of the following matrix
\begin{equation*}
\left( \begin{array}{cccccc}
0 & 0 & \epsilon^{2n-4}_i &\epsilon^{2n-4}_j\\ 
\frac{\partial C_{2n-3}}{\partial t_{2n-3}}&\frac{\partial C_{2n-2}}{\partial t_{2n-3}}& \epsilon^{2n-3}_i &\epsilon^{2n-3}_j\\ 
\frac{\partial C_{2n-3}}{\partial x_1}&\frac{\partial C_{2n-2}}{\partial x_1}  & \epsilon^{2n-2}_i &\epsilon^{2n-2}_j\\    
\frac{\partial C_{2n-3}}{\partial x_2}&\frac{\partial C_{2n-2}}{\partial x_2}& \epsilon^{2n-1}_i &\epsilon^{2n-1}_j\\
\frac{\partial C_{2n-3}}{\partial x_3}&\frac{\partial C_{2n-2}}{\partial x_3}&\epsilon^{2n}_i &\epsilon^{2n}_j
\end{array} 
\right)
\end{equation*}

which gives the coordinates of the Poisson bivector associated to $f_o$. 
\medskip

When $f$ is a real map, we take the coordinates functions as Casimir functions:
\begin{eqnarray*}
C_i &=&t_i \quad 1\leq i \leq 2n-3\\
C_{2n-2} &= &f_o(t_1, \dots, t_{2n-3}, x_1, x_2, x_3)
\end{eqnarray*}

The differential matrix of the map is
\begin{equation}\label{array:differential-Casimir-N-dimR}
\left( \begin{array}{cccccc}
1 &0 & 0 & 0& \frac{\partial C_{2n-2}}{\partial t_1}\\
\vdots &\ddots & \vdots & \vdots& \vdots\\
0 & \cdots & 1 & 0 & \frac{\partial C_{2n-2}}{\partial t_{2n-4}}\\ 
0 &\cdots &0& 1&\frac{\partial C_{2n-2}}{\partial t_{2n-3}}\\ 
0 &\cdots &0 &0&\frac{\partial C_{2n-2}}{\partial x_1} \\    
0 &\cdots &0&0&\frac{\partial C_{2n-2}}{\partial x_2}\\
0 &\cdots &0& 0 &\frac{\partial C_{2n-2}}{\partial x_3}
\end{array} 
\right)
\end{equation}

Then, the coefficients of the corresponding bivector matrix are given by
\begin{equation*}
\pi^{ij}=Det\left[
\left( \begin{array}{cccccccc}
1 &0 & 0 & 0& \frac{\partial C_{2n-2}}{\partial t_1}& \epsilon^1_i &\epsilon^{1}_j\\
\vdots &\ddots & \vdots & \vdots& \vdots& \vdots& \vdots\\
0 & \cdots & 1 & 0 & \frac{\partial C_{2n-2}}{\partial t_{2n-4}} & \epsilon^{2n-4}_i &\epsilon^{2n-4}_j\\ 
0 &\cdots &0& 1 &\frac{\partial C_{2n-2}}{\partial t_{2n-3}}& \epsilon^{2n-3}_i &\epsilon^{2n-3}_j\\ 
0 &\cdots &0 &0 &\frac{\partial C_{2n-2}}{\partial x_1}  & \epsilon^{2n-2}_i &\epsilon^{2n-2}_j\\    
0 &\cdots &0&0&\frac{\partial C_{2n-2}}{\partial x_2}& \epsilon^{2n-1}_i &\epsilon^{2n-1}_j\\
0 &\cdots &0& 0 &\frac{\partial C_{2n-2}}{\partial x_3}&\epsilon^{2n}_i &\epsilon^{2n}_j
\end{array} 
\right)\right]
\end{equation*}

We note that $\pi^{ij}=0$ for $1\leq i\leq 2n-4$ and $1\leq j\leq 2n-4$. The rest of the coefficients can be computed with the following
\begin{equation*}
Det\left[
\left( \begin{array}{ccccc}
1 &0 & \epsilon^{2n-3}_i &\epsilon^{2n-3}_j\\ 
0 &\frac{\partial C_{2n-2}}{\partial x_1}  & \epsilon^{2n-2}_i &\epsilon^{2n-2}_j\\    
0&\frac{\partial C_{2n-2}}{\partial x_2}& \epsilon^{2n-1}_i &\epsilon^{2n-1}_j\\
0 &\frac{\partial C_{2n-2}}{\partial x_3}&\epsilon^{2n}_i &\epsilon^{2n}_j
\end{array} 
\right)\right]
\end{equation*}

In fact, the only nonzero coefficients are:
\begin{eqnarray*}
\pi^{23}&=&\frac{\partial C_{2n-2}}{\partial x_3}\\
\pi^{24}&=&\frac{\partial C_{2n-2}}{\partial x_2}\\
\pi^{34}&=&\frac{\partial C_{2n-2}}{\partial x_3}
\end{eqnarray*}

The result follows.
\end{proof}
%
%

\subsection{Poisson structures on generalized wrinkled fibrations in dimension $6$.}\label{SS:Poisson-Sym-Wrinkled}
We apply the general criterion presented above to the case of wrinkled fibrations on 6-manifolds. Let $q\in M$ be a point, and $k:M\to X, k(t_1, t_2, t_3, x_1, x_2, x_3)$, be a non-vanishing smooth function. 

%
%
\subsubsection{Poisson bivector near a fold singularity.}\label{SSS:SingularFold}
\noindent {\bf Indefinite fold}
\\
The local coordinate model around a fold singularity is given by the map:
\begin{equation*}
(t_1, t_2, t_3, x_1, x_2, x_3) \mapsto (t_1, t_2, t_3, -x_1^2 + x_2^2 + x_3^2)
\end{equation*}
The resulting Poisson structure of a fold singularity is given by:
\begin{equation}\label{biv:fold}
\pi = k\left[ 2x_3\frac{\partial}{\partial x_2}\wedge\frac{\partial}{\partial x_1} - 2x_2 \frac{\partial}{\partial x_3}\wedge \frac{\partial}{\partial x_1} -2x_1\frac{\partial}{\partial x_3}\wedge \frac{\partial}{\partial x_2} \right] 
\end{equation}

\noindent {\bf Definite fold}
\\
In addition, we also compute the Poisson bivector for definite singularities for each wrinkled fibration. In this case, they are locally modeled by (\ref{eqn:fold-def1}) and (\ref{eqn:fold-def2}):
\begin{equation}\label{eqn:fold-def1}(t_1, t_2, t_3, x_1, x_2, x_3)\mapsto (t_1, t_2, t_3, x_1^2+x_2^2+x_3^2)
\end{equation}

\begin{equation}\label{eqn:fold-def2}(t_1, t_2, t_3, x_1, x_2, x_3)\mapsto (t_1, t_2, t_3, x_1^2+x_2^2+x_3^2)
\end{equation}
Following the general computations as above, the Poisson bivectors are, respectively:
\begin{equation}\label{biv:fold-def1}
\pi = k\left[ 2x_3\frac{\partial}{\partial x_2}\wedge\frac{\partial}{\partial x_1} - 2x_2 \frac{\partial}{\partial x_3}\wedge \frac{\partial}{\partial x_1} +2x_1\frac{\partial}{\partial x_3}\wedge \frac{\partial}{\partial x_2} \right] \end{equation}
and 

\begin{equation}\label{biv:fold-def2}
\pi = k\left[ -2x_3\frac{\partial}{\partial x_2}\wedge\frac{\partial}{\partial x_1} + 2x_2 \frac{\partial}{\partial x_3}\wedge \frac{\partial}{\partial x_1} +2x_1\frac{\partial}{\partial x_3}\wedge \frac{\partial}{\partial x_2} \right] \end{equation}

%
%
\subsubsection{Poisson bivector near a cusp singularity.}\label{SSS:SingularCusp}
\noindent {\bf Indefinite cusp}
\\
The local coordinate model around a cusp singularity is given by:
\begin{equation*}
(t_1, t_2, t_3, x_1, x_2, x_3) \mapsto (t_1, t_2, t_3, x_1^3 - 3 t_1x_1 + x_2^2 - x_3^2)
\end{equation*}
The Poisson bivector in the local coordinates of a cusp singularity is given by:
\begin{equation}\label{biv:cusp}
\pi = k\left[ -2x_3\frac{\partial}{\partial x_2}\wedge\frac{\partial}{\partial x_1} - 2x_2 \frac{\partial}{\partial x_3}\wedge \frac{\partial}{\partial x_1} +3(x_1^2-t_1)\frac{\partial}{\partial x_3}\wedge \frac{\partial}{\partial x_2} \right] \end{equation}

\noindent {\bf Definite cusp}
\\
For definite singularities in cusps, we obtain in each case (\ref{eqn:cusp-def1}) and (\ref{eqn:cusp-def2}):
\begin{equation}\label{eqn:cusp-def1}(t_1, t_2, t_3, x_1, x_2, x_3)\mapsto (t_1, t_2, t_3, x_1^3-3t_1x_1 +x_2^2+ x_3^2)
\end{equation}
\begin{equation}\label{eqn:cusp-def2}(t_1, t_2, t_3, x_1, x_2, x_3)\mapsto (t_1, t_2, t_3, x_1^3-3t_1x_1 -x_2^2- x_3^2)
\end{equation}
The corresponding bivectors are, respectively:
\begin{equation}\label{biv:cusp-def1}
\pi = k\left[ 2x_3\frac{\partial}{\partial x_2}\wedge\frac{\partial}{\partial x_1} - 2x_2 \frac{\partial}{\partial x_3}\wedge \frac{\partial}{\partial x_1} +3(x_1^2-t_1)\frac{\partial}{\partial x_3}\wedge \frac{\partial}{\partial x_2} \right] \end{equation}
and
\begin{equation}\label{biv:cusp-def2}
\pi = k\left[ -2x_3\frac{\partial}{\partial x_2}\wedge\frac{\partial}{\partial x_1} + 2x_2 \frac{\partial}{\partial x_3}\wedge \frac{\partial}{\partial x_1} +3(x_1^2-t_1)\frac{\partial}{\partial x_3}\wedge \frac{\partial}{\partial x_2} \right] \end{equation}

%
%
\subsubsection{Poisson bivector near a swallowtail singularity.}\label{SSS:SingularSwallowtail}
\noindent {\bf Indefinite swallowtail}
\\
The local coordinate model around a swallowtail singularity is given by the map:
\begin{equation*}
(t_1, t_2, t_3, x_1, x_2, x_3) \mapsto (t_1, t_2, t_3, x_1^4 + t_1 x_1^2 + t_2 x_1 + x_2^2 - x_3^2)
\end{equation*}
The Poisson bivector in the local coordinates of a swallowtail singularity is described by:
\begin{equation}\label{biv:swallowtail}
\pi = k\left[ -2x_3\frac{\partial}{\partial x_2}\wedge\frac{\partial}{\partial x_1} - 2x_2 \frac{\partial}{\partial x_3}\wedge \frac{\partial}{\partial x_1} +(4x_1^3+2t_1x_1+t_2)\frac{\partial}{\partial x_3}\wedge \frac{\partial}{\partial x_2} \right] 
\end{equation}

\noindent {\bf Definite swallowtail}
\\
For definite singularities:
\begin{equation}\label{eqn:swallowtail-def1}(t_1, t_2, t_3, x_1, x_2, x_3)\mapsto (t_1, t_2, t_3, x_1^4+t_1x_1^2+t_2x_1+x_2^2+x_3^2)
\end{equation}
\begin{equation}\label{eqn:swallowtail-def2}(t_1, t_2, t_3, x_1, x_2, x_3)\mapsto (t_1, t_2, t_3, x_1^4+t_1x_1^2+t_2x_1-x_2^2-x_3^2)
\end{equation}
The corresponding bivectors are, respectively:

\begin{equation}\label{biv:swallowtail-def1}
\pi = k\left[ 2x_3\frac{\partial}{\partial x_2}\wedge\frac{\partial}{\partial x_1} - 2x_2 \frac{\partial}{\partial x_3}\wedge \frac{\partial}{\partial x_1} +(4x_1^3+2t_1x_1+t_2)\frac{\partial}{\partial x_3}\wedge \frac{\partial}{\partial x_2} \right] 
\end{equation}
and 
\begin{equation}\label{biv:swallowtail-def2}
\pi = k\left[ -2x_3\frac{\partial}{\partial x_2}\wedge\frac{\partial}{\partial x_1} + 2x_2 \frac{\partial}{\partial x_3}\wedge \frac{\partial}{\partial x_1} +(4x_1^3+2t_1x_1+t_2)\frac{\partial}{\partial x_3}\wedge \frac{\partial}{\partial x_2} \right] 
\end{equation}

%
%
\subsubsection{Poisson bivector near a butterfly singularity.}\label{SSS:Singularbutterfly}
\noindent {\bf Indefinite butterfly}
\\
The local coordinate model around a buttterfly singularity is given by:
\begin{equation*}
(t_1, t_2, t_3, x_1, x_2, x_3) \mapsto (t_1, t_2, t_3, x_1^5 + t_1 x_1^3 + t_2 x_1^2 + t_3 x_1 + x_2^2 - x_3^2)
\end{equation*}
The Poisson bivector in the local coordinates of a butterfly singularity is described by:
\begin{equation}\label{biv:butterfly}
\pi = k\left[ -2x_3\frac{\partial}{\partial x_2}\wedge\frac{\partial}{\partial x_1} - 2x_2 \frac{\partial}{\partial x_3}\wedge \frac{\partial}{\partial x_1} +(5 x_1^4+3 t_1 x_1^2+2 t_2 x_1+t_3) \frac{\partial}{\partial x_3}\wedge \frac{\partial}{\partial x_2} \right] 
\end{equation}

\noindent {\bf Definite butterfly}
The singularity is modeled by the coordinates:
\begin{equation}\label{eqn:butterfly-def1}
(t_1, t_2, t_3, x_1, x_2, x_3) \mapsto (t_1, t_2, t_3, x_1^5 + t_1 x_1^3 + t_2 x_1^2 + t_3 x_1 + x_2^2 + x_3^2)\end{equation}

\begin{equation}\label{eqn:butterfly-def2}(t_1, t_2, t_3, x_1, x_2, x_3) \mapsto (t_1, t_2, t_3, x_1^5 + t_1 x_1^3 + t_2 x_1^2 + t_3 x_1 - x_2^2 - x_3^2)
\end{equation}
The corresponding bivectors are, respectively:
\begin{equation}\label{biv:butterfly-def1}
\pi = k\left[ 2x_3\frac{\partial}{\partial x_2}\wedge\frac{\partial}{\partial x_1} - 2x_2 \frac{\partial}{\partial x_3}\wedge \frac{\partial}{\partial x_1} +(5 x_1^4+3 t_1 x_1^2+2 t_2 x_1+t_3) \frac{\partial}{\partial x_3}\wedge \frac{\partial}{\partial x_2} \right] 
\end{equation}
and 
\begin{equation}\label{biv:butterfly-def2}
\pi = k\left[ -2x_3\frac{\partial}{\partial x_2}\wedge\frac{\partial}{\partial x_1} + 2x_2 \frac{\partial}{\partial x_3}\wedge \frac{\partial}{\partial x_1} +(5 x_1^4+3 t_1 x_1^2+2 t_2 x_1+t_3) \frac{\partial}{\partial x_3}\wedge \frac{\partial}{\partial x_2} \right] 
\end{equation}

%
%
%
\subsection{Poisson bivectors on higher dimensional type $2n$ generalized wrinkled fibrations.}

Lekili defined $4$ moves, these include all the possible $1$--parameter deformations of broken and wrinkled fibrations up to homotopy (see \cite{L09}). Lekili showed that any $1$-parameter family deformation of a purely wrinkled fibration is homotopic (relative endpoints) to one which realises a sequence of births, merges, flips, their inverses, and isotopies staying within the class of purely wrinkled fibrations. For higher dimensions, we will introduce a generalized form of these deformations. We will use them to give local expressions for the associated Poisson bivectors and symplectic forms near singularities described by the deformations.
\medskip

Consider the following maps $\mathbb{R} \times \mathbb {R}^{2n-1} \to \mathbb{R}^{2n-2}$, given by the equations below, and each depending on a real parameter $s$:

\begin{equation}\label{eqn:bs}
b_s(t_1, \dots, t_{2n-3}, x_1, x_2, x_3)=(t_1, \dots, t_{2n-3}, x_1^3-3x_1(t_{2n-3}^2-s)+x_2^2-x_3^2)
\end{equation}

\begin{equation}\label{eqn:ms}
m_s(t_1, \dots, t_{2n-3}, x_1, x_2, x_3)=(t_1, \dots, t_{2n-3}, x_1^3-3x_1(s-t_{2n-3}^2)+x_2^2-x_3^2)
\end{equation}
\begin{equation}\label{eqn:fs}
f_s(t_1, \dots, t_{2n-3}, x_1, x_2, x_3)=(t_1, \dots, t_{2n-3}, x_1^4-x_1^2s+x_1t_{2n-3}+x_2^2-x_3^2)
\end{equation}
\begin{equation}\label{eqn:ws}
w_s(t_1, \dots, t_{2n-3}, x_1, x_2, x_3)=(t_1, \dots, t_{2n-4}, t_{2n-3}^2-x_1^2+x_2^2-x_3^2+st_{2n-3}, 2t_{2n-3}x_1+2x_2x_3)
\end{equation}

We will also need a generalized winkled fibration for dimensions greater than $6$. 
\begin{definition}
Let $M$ be a smooth $2n$--manifold, and $X$ be a smooth closed $2n-2$--manifold. A type $2n$-wrinkled fibration is a smooth map $f:M\to X$ that is a sumbersion with the following four indefinite singularities each locally modelled by real charts $\mathbb{R}^{2n}\to\mathbb{R}^{2n-2}$
\begin{enumerate}

\item folds
$$(t_1, \dots, t_{2n-3}, x_1, x_2, x_3) \mapsto (t_1, \dots, t_{2n-3 }, -x_1+x_2^2+x_3^2)$$

 \item cusps 
$$(t_1, \dots, t_{2n-3}, x_1, x_2, x_3) \mapsto (t_1, \dots, t_{2n-3 }, x_1^3-3t_1\cdot x_1+x_2^2-x_3^2)$$

\item swallowtails
$$(t_1, \dots, t_{2n-3}, x_1, x_2, x_3) \mapsto (t_1, \dots, t_{2n-3 }, x_1^4+t_1x_1^2+t_2x_1+x_2^2-x_3^2)$$

\item butterflies
$$(t_1, \dots, t_{2n-3}, x_1, x_2, x_3) \mapsto (t_1, \dots, t_{2n-3 }, x_1^5+t_1x_1^3+t_2x_1^2+t_3x_1+x_2^2-x_3^2)$$
\end{enumerate}
\end{definition}

\begin{corollary}\label{Cor:localPoisson}
For a non-vanishing smooth function $k$ in $C^{\infty}(M)$ we have the following consequences:

\noindent {\em (1)} Let $X$ be a closed smooth oriented and connected $2n$-manifold, and $f : M \to X$ a generalized broken Lefschetz fibration. The Poisson structures in a neighborhood of the two type of singularities can be computed to obtain Poisson bivectors near the following singularities
\\
\noindent {\bf  Lefschetz-type singularity}

\begin{equation*}
\pi=k\left[(x_2^2+x_3^2)\frac{\partial}{\partial t_{2n-3}}\wedge \frac{\partial }{\partial x_1}+(x_1 x_2- t_{2n-3} x_3) \frac{\partial}{\partial t_{2n-3}}\wedge \frac{\partial}{\partial x_2}-(t_{2n-3} x_2+x_1 x_3) \frac{\partial}{\partial t_{2n-3}}\wedge \frac{\partial}{\partial x_3}\right.
\end{equation*}
\begin{equation*}
\quad \quad+\left.(t_{2n-3}x_2+x_1x_3)\frac{\partial}{\partial x_1}\wedge \frac{\partial}{\partial x_2}+(x_1x_2- t_{2n-3}x_3)\frac{\partial}{\partial x_1}\wedge \frac{\partial}{\partial x_3}+( t_{2n-3}^2+x_1^2)\frac{\partial}{\partial x_2}\wedge \frac{\partial}{\partial x_3}\right]
\end{equation*}

\noindent {\bf  Indefinite fold singularity}
\begin{equation*}
\pi=k\left[x_1\frac{\partial}{\partial x_2}\wedge\frac{\partial}{\partial x_3}+x_2\frac{\partial }{\partial x_1}\wedge\frac{\partial}{\partial x_3}-x_3\frac{\partial}{\partial x_1}\wedge\frac{\partial }{\partial x_2}\right]
\end{equation*}

\noindent {\em (2)} Let $M$ be a closed, orientable, smooth $2n$-manifold endowed with a type $2n$-wrinkled fibration $f$ to a closed $2n-2$ manifold $X$. Then a complete Poisson structure is given by the following bivectors near the corresponding singularities:
\medskip

\noindent {\bf Fold}
\begin{equation*}
\pi = k\left[ 2x_3\frac{\partial}{\partial x_2}\wedge\frac{\partial}{\partial x_1} - 2x_2 \frac{\partial}{\partial x_3}\wedge \frac{\partial}{\partial x_1} -2x_1\frac{\partial}{\partial x_3}\wedge \frac{\partial}{\partial x_2} \right] \end{equation*}

\noindent {\bf Cusp}
\begin{equation*}
\pi = k\left[ -2x_3\frac{\partial}{\partial x_2}\wedge\frac{\partial}{\partial x_1} - 2x_2 \frac{\partial}{\partial x_3}\wedge \frac{\partial}{\partial x_1} +3(x_1^2-t_{2n-5})\frac{\partial}{\partial x_3}\wedge \frac{\partial}{\partial x_2} \right] \end{equation*}

\noindent {\bf Swallowtail}
\begin{equation*}
\pi = k\left[ -2x_3\frac{\partial}{\partial x_2}\wedge\frac{\partial}{\partial x_1} - 2x_2 \frac{\partial}{\partial x_3}\wedge \frac{\partial}{\partial x_1} +3(4x_1^3+2t_{2n-5}x_1+t_{2n-4})\frac{\partial}{\partial x_3}\wedge \frac{\partial}{\partial x_2} \right] 
\end{equation*}

\noindent {\bf Butterfly}
\begin{equation*}
\pi = k\left[ -2x_3\frac{\partial}{\partial x_2}\wedge\frac{\partial}{\partial x_1} - 2x_2 \frac{\partial}{\partial x_3}\wedge \frac{\partial}{\partial x_1} +(5 x_1^4+3 t_{2n-5} x_1^2+2 t_{2n-4} x_1+t_{2n-3}) \frac{\partial}{\partial x_3}\wedge \frac{\partial}{\partial x_2} \right] 
\end{equation*}

\noindent The following equations depend on a real parameter $s$. Near a singularity locally modeled by the $b_s, m_s, f_s,$ and $w_s$ the corresponding Poisson bivectors are 

\noindent {\bf Map}  $b_s$ 
\begin{equation*}\pi_s = k\left[ 2x_3\frac{\partial}{\partial x_1}\wedge\frac{\partial}{\partial x_2} + 2x_2 \frac{\partial}{\partial x_1}\wedge \frac{\partial}{\partial x_3} - 3(s-t_{2n-3}^2+x_1^2)\frac{\partial}{\partial x_2}\wedge \frac{\partial}{\partial x_3} \right] \end{equation*}

\noindent {\bf Map}  $m_s$ 
\begin{equation*}\pi_s =k\left[ 2x_3\frac{\partial}{\partial x_1}\wedge\frac{\partial}{\partial x_2} + 2x_2\frac{\partial}{\partial x_1}\wedge \frac{\partial}{\partial x_3} - 3(s-t_{2n-3}^2-x_1^2)\frac{\partial}{\partial x_2}\wedge \frac{\partial}{\partial x_3} \right] \end{equation*}

\noindent {\bf Map}  $f_s$ 
\begin{equation*}\pi_s = k\left[2x_3\frac{\partial}{\partial x_1}\wedge\frac{\partial}{\partial x_2} + 2x_2\frac{\partial}{\partial x_1}\wedge \frac{\partial}{\partial x_3} - (t_{2n-3}-2sx_1+4x_1^3)\frac{\partial}{\partial x_2}\wedge \frac{\partial}{\partial x_3} \right] \end{equation*}

\noindent  {\bf Map} $w_s$  
\begin{equation*}\pi_s  = k \left[ (-2 s x_2 - 4 t_{2n-3} x_2 - 4 x_1 x_3)\frac{\partial}{\partial x_1}\wedge\frac{\partial}{\partial x_2} + (-4 x_1 x_2 + 2 s x_3 + 4 t_{2n-3}x_3)\frac{\partial}{\partial x_1}\wedge \frac{\partial}{\partial x_3}  \right.\end{equation*}
\begin{equation*}
\quad \quad +(4 x_2^2 + 4 x_3^2)\frac{\partial}{\partial x_1}\wedge \frac{\partial}{\partial t_{2n-3}} - (2 s t_{2n-3} + 4 t_{2n-3}^2 + 4 x_1^2) \frac{\partial}{\partial x_2}\wedge \frac{\partial}{\partial x_3}
\end{equation*}
\begin{equation*} \quad \quad \left. + 4(x_1 x_2 - t_{2n-3} x_3) \frac{\partial}{\partial x_2}\wedge \frac{\partial}{\partial t_{2n-3}} -4(t_{2n-3} x_2 + x_1 x_3)\frac{\partial}{\partial x_3}\wedge \frac{\partial}{\partial t_{2n-3}}  \right]\end{equation*}

\end{corollary}

The proof of this last result is included in appendix A.

%
%
\section{Symplectic forms on leaves of generalized wrinkled fibrations}\label{S:symp-forms}

\subsection{General criterion for constructing symplectic forms on leaves near the singularities}

\begin{theorem}
Under the hypothesis of Proposition \ref{bivector-criterion}, the symplectic form induced by the Poisson structure $\pi$ on the symplectic leaf $\Sigma_q$ through $q\neq 0$ is completely determined by the Poisson structure of the  map $f_o$. That is, if $u_q, v_q$ are tangent vectors to the leaves, then:
\begin{equation*}
\omega_{\Sigma_q}(u_q, v_q)=\omega_o(\tilde u_q, \tilde v_q)
\end{equation*}

where $\omega_o$ is the symplectic structure of $f_o$, and $\tilde u_q, \tilde v_q$ are the tangent vectors $u_q$ and $v_q$ restricted to the last 4 coordinates.
\end{theorem}

\begin{proof}
First, we have to obtain vectors tangent to the leaves. That is, we want to find vectors such that they are annhilated simultaneously by the $2n-2$ Casimir functions. Then we transpose the matrix and compute its null space. 
\medskip

In the case when $f$ is a complex map, we used its real and imaginary parts of each coordinate function as Casimir functions. We obtained the matrix (\ref{array:differential-Casimir-N-dimC}) whose transpose matrix is:
\begin{equation*}
D^{T}=\left( \begin{array}{cccccccc}
1 &\cdots &0 &0 &0 &0 &0\\
\vdots &\ddots & \vdots & \vdots& \vdots& \vdots& \vdots\\
0 & \cdots & 1 & 0 & 0 &0 &0\\ 
0 &\cdots &0& \frac{\partial C_{2n-3}}{\partial t_{2n-3}}& \frac{\partial C_{2n-3}}{\partial x_1}& \frac{\partial C_{2n-3}}{\partial x_2}& \frac{\partial C_{2n-3}}{\partial x_3}\\ 
0 &\cdots &0 &\frac{\partial C_{2n-2}}{\partial t_{2n-3}}&\frac{\partial C_{2n-2}}{\partial x_1} & \frac{\partial C_{2n-2}}{\partial x_2}& \frac{\partial C_{2n-2}}{\partial x_3}
\end{array} 
\right)
\end{equation*}

Note that its left upper block is  an identity matrix of dimension $2n-4$.

Let 
\begin{eqnarray*}
\partial C_{2n-3}:&=& \left(0, \dots 0, \frac{\partial C_{2n-3}}{\partial t_{2n-3}}, \frac{\partial C_{2n-3}}{\partial x_1}, \frac{\partial C_{2n-3}}{\partial x_2}, \frac{\partial C_{2n-3}}{\partial x_3}\right)\\
\partial C_{2n-3}:&=&  \left(0, \dots 0, \frac{\partial C_{2n-2}}{\partial t_{2n-3}}, \frac{\partial C_{2n-2}}{\partial x_1}, \frac{\partial C_{2n-2}}{\partial x_2}, \frac{\partial C_{2n-2}}{\partial x_3}\right)
\end{eqnarray*}

Then a vector $a=(a_1, a_2, \dots, a_{2n})$ belongs to  $Ker(D^T)$  if and only if:
\begin{eqnarray*}
\langle \partial C_{2n-3}, a\rangle=0\\
\langle \partial C_{2n-2}, a\rangle=0
\end{eqnarray*}

Observe that the first $2n-4$ entries of $a$ equal zero. Then, $a\in Ker (D^T)$ if 
\begin{equation*}
a=(0, 0, \dots, 0, a_{2n-3}, a_{2n-2}, a_{2n-1}, a_{2n}),
\end{equation*}
 where the coefficients $a_{2n-3}, a_{2n-2}, a_{2n-1}, a_{2n}$ are determined by the equations:
\begin{equation}\label{eqn:kernel-generalized}
\left\{\begin{array}{lcl}
a_{2n-3} \frac{\partial C_{2n-3}}{\partial t_{2n-3}}+a_{2n-2} \frac{\partial C_{2n-3}}{\partial x_1}+ a_{2n-1}\frac{\partial C_{2n-3}}{\partial x_2}+ a_{2n}\frac{\partial C_{2n-3}}{\partial x_3}=0\\
a_{2n-3} \frac{\partial C_{2n-2}}{\partial t_{2n-3}}+a_{2n-2} \frac{\partial C_{2n-2}}{\partial x_1}+ a_{2n-1}\frac{\partial C_{2n-2}}{\partial x_2}+ a_{2n}\frac{\partial C_{2n-2}}{\partial x_3}=0
\end{array}\right.
\end{equation}

Since the rank of the matrix $D$ is $2n-2$, it has nullity $2$. Therefore there exist two vectors $u_q$ and $v_q$ that  generate all solutions to the previous system. We may assume they are orthogonal. Now, we have to find vectors $\alpha_q, \beta_q$ such that $\mathcal{B}_q(\alpha_q)=u_q$ and $\mathcal{B}_q(\beta_q)=v_q$.
\medskip

To compute the symplectic form it is enough to find $\alpha_q$. In order to compute $\beta_q$ we may proceed similarly. We know that $\alpha$ is the solution to the equation $\mathcal{B_q}(\alpha)(\cdot)=\pi(\cdot, \alpha)=u_q$. 

It is equivalent to consider the system $\pi\cdot\alpha_q=u_q$ and solve for $\alpha_q$. By the previous discussion and recalling the form of the Poisson matrix, if $u_q$, $\alpha_q$ and $v_q$ have coordinates:
\begin{eqnarray*}
u_q&=&(0, 0, \dots, u_{2n-3}, u_{2n-2}, u_{2n-1}, u_{2n})\\
v_q&=&(0, 0, \dots, u_{2n-3}, v_{2n-2}, v_{2n-1}, v_{2n})\\
\alpha_q&=&(\alpha_1, \alpha_2, \dots, \alpha_{2n})
\end{eqnarray*}

This system is reduced to:
\begin{equation}\label{eqn:symplectic-alpha}
\left\{ \begin{array}{lcl}
u_{2n-3}&=&\alpha_{2n-2}\pi^{12}+\alpha_{2n-1}\pi^{13}+\alpha_{2n}\pi^{14}\\
u_{2n-2}&=&-\alpha_{2n-3}\pi^{12}+\alpha_{2n-1}\pi^{23}+\alpha_{2n}\pi^{24}\\
u_{2n-1}&=&-\alpha_{2n-3}\pi^{13}-\alpha_{2n-2}\pi^{23}+\alpha_{2n}\pi^{34}\\
u_{2n}&=&-\alpha_{2n-3}\pi^{14}-\alpha_{2n-2}\pi^{24}-\alpha_{2n-1}\pi^{34}
\end{array}\right.
\end{equation}

Therefore the symplectic form will be given by
\begin{equation*}
\omega_{\Sigma_q}(q)=\langle\alpha_q, v_q\rangle,
\end{equation*}
here $\alpha_q$ is the solution to the system (\ref{eqn:symplectic-alpha}), and $v$ satisfies the system (\ref{eqn:kernel-generalized}). Note that we may choose $\alpha$ with the first $2n-4$ coordinates equal zero.
\medskip

When the map $f$ is real we obtained the matrix (\ref{array:differential-Casimir-N-dimR}). Its transpose is:

\begin{equation*}
D^T=\left( \begin{array}{cccccccc}
1 &\cdots &0 &0 &0 &0 &0\\
0 &\ddots & \vdots & \vdots& \vdots& \vdots& \vdots\\
0 & \cdots & 1 & 0 & 0 &0 &0\\ 
0 &\cdots &0& 1& 0&0&0\\
\frac{\partial C_{2n-2}}{\partial t_1}& \cdots& \frac{\partial C_{2n-2}}{t_{2n-4}}& \frac{\partial C_{2n-2}}{t_{2n-3}}&\frac{\partial C_{2n-2}}{x_1}& \frac{\partial C_{2n-2}}{x_2}& \frac{\partial C_{2n-2}}{x_3}
\end{array} 
\right)
\end{equation*}

Its left upper block is  an identity matrix of dimension $2n-3$. Then $a\in Ker (D^T)$ if $a=(0, 0, \dots, 0, 0, a_{2n-2}, a_{2n-1}, a_{2n})$, where the coefficients $a_{2n-2}, a_{2n-1}, a_{2n}$ are determined by the equation:
\begin{equation*}
a_{2n-2}\frac{\partial C_{2n-2}}{x_1}+ a_{2n-1}\frac{\partial C_{2n-2}}{x_2}+ a_{2n}\frac{\partial C_{2n-2}}{x_3}=0
\end{equation*}

We can give the explicit solutions, they are generated by the vectors:
\begin{equation}\label{eqn:uqvq-expression-real}
u=\{0, 0, \dots, 0, -\frac{\frac{\partial C_{2n-2}}{x_2}}{\frac{\partial C_{2n-2}}{x_1}}, 1, 0\},\\
v=\{0, 0, \dots, 0, -\frac{\frac{\partial C_{2n-2}}{x_3}}{\frac{\partial C_{2n-2}}{x_1}}, 0, 1\}
\end{equation}

Let $u_q=u$ and $v_q=proj_{u}(v)$, the orthogonal projection of $v$ over $u$. Then  $u_q$ and $v_q$ are orthogonal and generate all solutions to the previous system. As before, we know that $\alpha_q$ is the solution to the equation $\mathcal{B_q}(\alpha)(\cdot)=\pi(\cdot, \alpha)=u_q$. 

This is equivalent to solving the system $\pi\cdot\alpha_q=u_q$ for $\alpha_q$. If  $\alpha_q$  has coordinates:
\begin{equation*}
\alpha_q=(\alpha_1, \alpha_2, \dots, \alpha_{2n})
\end{equation*}

this system is reduced to
\begin{equation}\label{eqn:symplectic-alpha-real}
\left\{ \begin{array}{lcl}
-\frac{\frac{\partial C_{2n-2}}{x_2}}{\frac{\partial C_{2n-2}}{x_1}}&=&-\alpha_{2n-3}\pi^{12}+\alpha_{2n-1}\pi^{23}+\alpha_{2n}\pi^{24}\\
1&=&-\alpha_{2n-3}\pi^{13}-\alpha_{2n-2}\pi^{23}+\alpha_{2n}\pi^{34}\\
0&=&-\alpha_{2n-3}\pi^{14}-\alpha_{2n-2}\pi^{24}-\alpha_{2n-1}\pi^{34}
\end{array}\right.
\end{equation}

Therefore the symplectic form will be given by
\begin{equation*}
\omega_{\Sigma_q}(q)=\langle\alpha_q, v_q\rangle
\end{equation*}

where $\alpha_q$ is the solution to the system (\ref{eqn:symplectic-alpha-real}), and $v_q$  has the form (\ref{eqn:uqvq-expression-real}). Note that we may choose $\alpha$ with the first $2n-4$ coordinates equal zero.\end{proof}

\subsection{Symplectic forms on the leaves of generalized wrinkled fibrations in dimension $6$}
As a corollary of the previous theorem we obtain the following result in dimension 6. 

\begin{corollary}\label{Cor:LocalSymplectic1}.
Let $M$ be a closed, orientable, smooth $6$--manifold equipped with a generalized wrinkled fibration $f\colon M\rightarrow X$ on a smooth 4-manifold $X$. Let $(U, (t_1, t_2, t_3, x_1, x_2, x_3))$ be a coordinate neighbourhood of $q\in \textnormal{Crit}_f$, an element of the singularity set of $f$ . Then, there is a symplectic form on $U$ induced by $\pi$ on the symplectic leaf $\Sigma_q$ through $q$  near each of the singularities of the fibration with the following expressions:

\noindent {\bf Indefinite Fold}
\begin{equation}\label{FoldForm}
\omega_{\Sigma_q}=\frac{x_1^2}{2 k(q)(x_1^2+x_3^2)^{1/2}}\omega_{Area}(q)
\end{equation}

\noindent where $\omega_{Area}$ is the area form on $\Sigma_q$ induced by the euclidean metric on $B^6$.

\noindent {\bf Definite Folds}
\\
For the definite definite fold singularities described by the equations (\ref{eqn:fold-def1}) and (\ref{eqn:fold-def2}) we obtain the symplectic forms
\begin{equation}\label{FoldFormDefinite1}
\omega_{\Sigma_q}=-\frac{x_1^2}{2 (x_1^2+x_3^2})^{1/2}\omega_{Area}(q)
\end{equation}
and
\begin{equation}\label{FoldFormDefinite2}
\omega_{\Sigma_q}=\frac{x_1^2}{2 (x_1^2+x_3^2})^{1/2}\omega_{Area}(q)
\end{equation}
respectively.

\noindent {\bf Indefinite Cusp}
\begin{equation}\label{CuspForm}
\omega_{\Sigma_q}=\frac{3 x_2 \left(t_1-x_1^2\right)}{k(q)(9 \left(t_1-x_1^2\right)^2+4 x_3^2)^{1/2}}\omega_{Area}(q)
\end{equation}
\noindent where $\omega_{Area}$ is the area form on $\Sigma_q$ induced by the euclidean metric on $B^6$.
\\
\noindent {\bf Definite Cusps}
\\
The definite singularities modelled by the parametrizations (\ref{eqn:cusp-def1}) and (\ref{eqn:cusp-def2}) have the corresponding symplectic form which coincides in both cases :
\begin{equation}\label{CuspFormDefinite}
\omega_{\Sigma_q}=\frac{3 (t_1 - x_1^2) x2}{(9 (t_1 - x_1^2)^2 + 4 x_3^2)^{1/2}}\omega_{Area}(q)
\end{equation}

\noindent {\bf Indefinite Swallowtail}
\begin{equation}\label{SwallowtailForm}
\omega_{\Sigma_q}=-\frac{t_2 + 2 t_1 x_1 + 4 x_1^3}{k(q)((t_2 + 2 t_1 x_1 + 4 x_1^3)^2 + 4 x_3^2)^{1/2}}\omega_{Area}(q)
\end{equation}

\noindent here $\omega_{Area}$ is the area form on $\Sigma_q$ induced by
the euclidean metric on $B^6$.

\noindent {\bf Definite Swallowtail}

The definite swallowtails modelled by the parametrizations  (\ref{eqn:swallowtail-def1}) and (\ref{eqn:swallowtail-def1}) have the corresponding symplectic form which coincides in both cases:
\begin{equation}\label{SwallowtailFormDefinite}
\omega_{\Sigma_q}=-\frac{(t_2 + 2 t_1 x_1 + 4 x_1^3}{((t_2 + 2 t_1 x_1 + 4 x_1^3)^2 + 4 x_3^2)^{1/2}}\omega_{Area}(q)
\end{equation}

\noindent {\bf Indefinite Butterfly}
\begin{equation}\label{ButterflyForm}
\omega_{\Sigma_q}=-\frac{t_3 + x_1 (2 t_2 + 3 t_1 x_1 + 5 x_1^3)}{k(q) ((t_3 + x_1 (2 t_2 + 3 t_1 x_1 + 5 x_1^3))^2 + 4 x_3^2)^{1/2}}\omega_{Area}(q)
\end{equation}

\noindent where $\omega_{Area}$ is the area form on $\Sigma_q$ induced by
the euclidean metric on $B^6$.

\noindent {\bf Definite Butterfly}

\begin{equation}\label{ButterflyFormDefinite}
\omega_{\Sigma_q}=-\frac{(t_3 + x_1 (2 t_2 + 3 t_1 x_1 + 5 x_1^3)}{(t_3 + 
    x_1 (2 t_2 + 3 t_1 x_1 + 5 x_1^3))^2 + 4 x_3^2)^{1/2}}\omega_{Area}(q)
\end{equation}

\end{corollary}

A proof may be found in appendix A. 

%
%
\subsection{Symplectic forms on higher dimensional type $2n$ generalized wrinkled fibrations}

\begin{corollary}\label{Cor:LocalSymplectic2}
For a non-vanishing smooth function $k\in C^{\infty}(M)$ we have the following consequences:
\\
\noindent {\em (1)} Let $M$ be a closed smooth oriented and connected $2n$-manifold, and $f : M \to X$ a generalized broken Lefschetz fibration. The symplectic forms induced by the corresponding Poisson structures on the symplectic leaves $\Sigma_q$ through  a point $q=(t_1, \dots, t_{2n-3}, x_1, x_2, x_3)$ have the following local expressions: 
\medskip 

\noindent {\bf  Lefschetz-type singularity}
\\
Let $q\in B^{2n}\backslash \{0\}$. Near Lefschetz-type singularities the symplectic form is given by 
\begin{equation*}
\omega_{\Sigma_q}= \frac{1}{k(q)(t_{2n-3}^2+x_1^2+x_2^2+x_3^2)}\omega_{Area}(p)
\end{equation*}

\noindent {\bf  Indefinite fold singularity}
\\
Near indefinite fold singularities $Z$ the symplectic form is locally described by
\begin{equation*}
\omega_{\Sigma_q}=\frac{1}{k(q)\sqrt{x_1^2+x_2^2+x_3^2}}\omega_{Area}(q)
\end{equation*}
where $\omega_{Area}(q)$ is the area form on $\Sigma_q$ induced by the metric

\begin{equation*}
ds^2=dt_1^2+\dots +dt_{2n-3}^2+dx_1^2+dx_2^2+dx_3^2
\end{equation*}
on $Z\times B^{3}$.

\medskip

\noindent {\em (2)} Let $M$ be a closed, orientable, smooth $2n$-manifold endowed with a type $2n$-wrinkled fibration $f$ to a closed $2n\!-\!2$ manifold $X$. Let $q\in B^{2n}\backslash \{0\}$. Then the symplectic forms associated to the complete Poisson structure are given by the following expressions near the corresponding singularities:
\medskip

\noindent {\bf Fold}
\begin{equation*}
\omega_{\Sigma_q}=\frac{x_1^2}{2 k(q)(x_1^2+x_3^2)^{1/2}}\omega_{Area}(q)
\end{equation*} 

\noindent {\bf Cusp}
\begin{equation*}
\omega_{\Sigma_q}=\frac{3 x_2 \left(t_{2n-5}-x_1^2\right)}{k(q)(9 \left(t_1-x_1^2\right)^2+4 x_3^2)^{1/2}}\omega_{Area}(q)
\end{equation*}

\noindent {\bf Swallowtail}
\begin{equation*}
\omega_{\Sigma_q}=-\frac{t_{2n-4} + 2 t_{2n-5} x_1 + 4 x_1^3}{k(q)((t_{2n-4} + 2 t_{2n-5} x_1 + 4 x_1^3)^2 + 4 x_3^2)^{1/2}}\omega_{Area}(q)
\end{equation*}

\noindent {\bf Butterfly}
\begin{equation*}
\omega_{\Sigma_q}=-\frac{t_{2n-3} + x_1 (2 t_{2n-4} + 3 t_{2n-5} x_1 + 5 x_1^3)}{k(q) ((t_{2n-3} + x_1 (2 t_{2n-4} + 3 t_{2n-5} x_1 + 5 x_1^3))^2 + 4 x_3^2)^{1/2}}\omega_{Area}(q)
\end{equation*}

\medskip
 The following equations depend on a real parameter $s$. Near a singularity locally modeled by the maps $b_s$  \eqref{eqn:bs}, $m_s$  \eqref{eqn:ms}, $f_s  \eqref{eqn:fs},$ and $w_s$ \eqref{eqn:ws}  the corresponding symplectic forms are
\\
\noindent {\bf Map}  $b_s$ 
\begin{equation*}
\omega_{\Sigma_q}=\frac{(s - t_{2n-3}^2 + x_1^2)}{k(q)((s - t_{2n-3}^2 +
x_1^2)^2(9 (s - t_{2n-3}^2 + x_1^2)^2 + 4 (x_2^2 +
x_3^2)))^{1/2}}\omega_{Area}(q)
 \end{equation*}

\noindent {\bf Map}  $m_s$ 
\begin{equation*}
\omega_{\Sigma_q}=-\frac{(s - t_{2n-3}^2 - x_1^2)}{k(q)((s - t_{2n-3}^2 -
x_1^2)^2(9 (s - t_{2n-3}^2 - x_1^2)^2 + 4 (x_2^2 +
x_3^2)))^{1/2}}\omega_{Area}(q)
\end{equation*}

\noindent {\bf Map}  $f_s$ 
\begin{equation*}
\omega_{\Sigma_q}=\frac{(t_{2n-3} - 2 s x_1 + 4 x_1^3)}{k(q)((t_{2n-3} - 2
s x_1 + 4 x_1^3)^2 ((t_{2n-3} - 2 s x_1 + 4 x_1^3)^2 + 4 (x_2^2 +
x_3^2)))^{1/2}}\omega_{Area}(q)
\end{equation*}

\noindent {\bf Map}  $w_s$ 
\begin{eqnarray*}
\omega_{\Sigma_q}&=&\frac{1}{2\mu k(q)}\cdot \\
&&\frac{(t_{2n-3} x_2 +
   x_1 x_3)((s t_{2n-3} + 2 (t_{2n-3}^2 + x_1^2))^2 + (x_3(s + 2 t_{2n-3})-2 x_1 x_2)^2 + 4 (t_{2n-3} x_2 + x_1 x_3))}{((s t_{2n-3} + 2 (t_{2n-3}^2 + x_1^2))^2 + (x_3(s + 2 t_{2n-3})-2 x_1 x_2)^2 + 4 (t_{2n-3} x_2 + x_1 x_3)^2)^{1/2}}\omega_{Area}(q)
\end{eqnarray*}

\noindent here $\omega_{Area}$ is the area form on $\Sigma_q$ induced by
the euclidean metric on $B^{2n}$, and
\begin{equation*}
\begin{split}
\mu^2=&(t_{2n-3} x_2 + x_1 x_3)^2 (s^2 (t_{2n-3}^2 + x_2^2 + x_3^2) + 4 s t_{2n-3} (t_{2n-3}^2 + x_1^2 +x_2^2
+x_3^2) + 4 (t_{2n-3}^2 + x_1^2 + x_2^2 + x_3^2)^2)\\
&(s^2 (t_{2n-3}^2 + x_3^2) + 4 (t_{2n-3}^2 + x_1^2) (t_{2n-3}^2 + x_1^2 + x_2^2 + x_3^2) +
   4 s (t_{2n-3}^3 - x_1 x_2 x_3 + t_{2n-3} (x_1^2 + x_3^2))).
\end{split}
\end{equation*}

\noindent For all these cases $\omega_{Area}(q)$ is the area form induced by the euclidean metric on $B^{2n}$.
\end{corollary}

We include a proof in appendix A.

\section{Near-symplectic Forms on Generalized Wrinkled Fibrations}\label{sec:near-sympletic}
\subsection{Near-symplectic Manifolds}
We follow the definition of near-symplectic forms in higher dimensions as in \cite{V16}. Let $M$ be an oriented manifold of dimension $2n$, and consider a 2-form $\omega \in \Omega^2(M)$ such that it is near-positive everywhere, that is  $\omega^n \geq  0$. Denote by $K_p = \lbrace v\in T_pM \, \mid \, \omega_p(v, \cdot) = 0 \rbrace$  the kernel of $\omega$ at a point $p\in M$. The collection of fibrewise kernels form the kernel of the 2-form  $K: = \ker (\omega) \subset TM$. If $\omega$ is symplectic, then $K_p = 0$. We relax the non-degeneracy condition by near-positive forms, we can consider non-trivial kernels $K_p$. There is an intrinsic gradient $\nabla_p \omega\colon K_p \rightarrow \Lambda^2 T_{p}^{*} M$. Restricting the gradient to bivectors in $K_p$ results in a linear map 
\begin{equation*}
D_K:= \nabla_p \omega|_{K} \colon K_p \rightarrow \Lambda^2 K^*.
\end{equation*}
The image $\textnormal{Im} (D_K)$ has dimension at most 3. Assuming that $K$ is 4-dimensional and Rank$(D_K)=3$ it has been shown in  \cite{V16} that the zero set of $\omega^{n-1}$ is a submanifold of $M$ of dimension $2n-3$. 

\begin{definition}\label{near-symplectic}
A 2-form $\omega \in \Omega^2 (M^{2n})$ is {\em near-symplectic}, if it is closed, $\omega^n \geq 0$, and at a point $p$ where $\omega^n = 0$, one has that the kernel $K$ is 4-dimensional and that the image Im$(D_K)$ has dimension 3.  
\\
The set $Z_{\omega} = \{ p\in M \mid \omega_{p}^{n-1} = 0 \}$ is called the {\em singular locus} of $\omega$ and it is a submanifold of codimension 3.
\end{definition}
\begin{remark}
In dimension 6, the definition of a near-symplectic form implies that $\omega \in \Omega^2(M)$ is closed and for every $p\in M$, either
\begin{itemize}
 \item[(i)] $\omega_p^3 > 0$ on $M\setminus Z_{\omega}$, or
 \item[(ii)] $\omega_p^{2} = 0$ on a 3-submanifold $Z_{\omega}$. 
\end{itemize}
\end{remark}

Locally, a Darboux-type theorem for near-symplectic forms tells us that we can find a coordinate neighbourhood $U$ around a point $p\in Z_{\omega} \subset (M, \omega)$ such that $\omega$ looks like the sum of a symplectic form of rank $2n-4$ and a 4-dimensional near-symplectic form. On $(U, (z,x))$ with coordinates $z = (z_0, \dots, z_{2n-3})$ on $Z_{\omega}$ and normal coordinates $x = (x_1, x_2, x_3)$, we can express $\omega$ locally as 
\begin{eqnarray*}
\omega &=& \omega_Z - 2 x_1 (dz_0 \wedge dx_1 + dx_2 \wedge dx_3 ) +   x_2 (dz_0 \wedge dx_2 - dx_1 \wedge dx_3 ) + x_3 (dz_0 \wedge dx_3 + dx_1 \wedge dx_2) 
\\
&=& \omega_Z  - 2x_1 (\beta_1) + x_2 (\beta_2) + x_3 (\beta_3)
\end{eqnarray*}
where $\omega_Z:= i^{*}\omega$ is a closed 2-form of maximal rank on $Z_{\omega}$.  On a 6-manifold, $\omega_Z$ would be of rank 2. The 2-forms $\beta_1, \beta_2$ and $\beta_3$ correspond to elements of a basis of the rank-3 bundle $\Lambda_{+}^2 \mathbb{R}^4$.

\subsection{Proof of Theorem \ref{thm:near-symplectic}} \label{sec:proof-near-symplectic}
\begin{theorem}
Let $M$ be a closed oriented 6-manifold, $(X, \omega_X)$ a closed symplectic 4-manifold, and $f\colon M \rightarrow X$ a generalized wrinkled fibration. Denote by $Z$ the singularity set of $f$, a 3-submanifold of $M$. Assume that there is a class $\alpha \in H^2(M)$, such that it pairs positively with every component of every fibre, and $\alpha|_{Z} = [\omega_X|_{Z}]$.  Then there exist a near-symplectic form $\omega$ on $M$ with singular locus $Z$ such that it restricts to a symplectic form on the smooth fibres of the fibration.
\end{theorem}

\begin{proof}
The global construction of a near-symplectic form on a generalized wrinkled fibration is similar to the 4-dimensional case. Constructing a near-symplectic form on the total space of a broken Lefschetz fibration involves four steps \cite{ADK05} that extend to the case of a wrinkled fibration \cite{L09}. These steps appear again in the higher dimensional situation with generalized BLFs \cite{V16}. We briefly recall them. Step 1 constructs a local near-symplectic that is positive on the fibres. Steps 2 and 3 extend the 2-form to the neighbourhood of the fibres and then to the whole manifold using the cohomological assumptions of the theorem. Finally, step 4 involves Thurston's argument to guarantee positivity on vertical and tangent subspaces. All these steps apply in the same way for generalized wrinkled fibrations. The only modification involves the local model of the 2-form around the new singularities.  Once this is done, there is no difference anymore in the global construction. Since this adjustment applies to step 1, we give the local near-symplectic forms for each singularity. 

\subsection{Constructing near-symplectic forms. General Scheme}
We begin by giving the general scheme to construct local near-symplectic forms around the critical set of $f$ without coordinates. The specific formul{\ae} in coordinates will be provided afterwards. To start, consider the following 2-form
\begin{equation}\label{omega0-general}
\omega_0 =  f^{*} \omega_X + \ast \left[ f^{*} \omega_X^2 \right] = du \wedge ds + dt \wedge df_4 + \ast( du \wedge ds \wedge dt \wedge df_4)
\end{equation}

where $\ast\colon \Omega^4M \rightarrow \Omega^2 M$ denotes the Hodge operator with respect to a Riemannian metric $g$ on $M$, and $df_4$ is the 1-form defined by the fourth component of the generalized wrinkled fibration which varies according to the parametrization of each singularity.  This 2-form is positive on the fibres and non-degenerate outside the singularity set by construction.  
The positivity on the fibres follows from $\ast \left[ f^{*} \omega_X^2 \right]$, since this 2-form is positive on the vertical subspaces, complementary to horizontal subspaces where the pullback $f^{*} \omega_X $ is positive. The non-degeneracy can be checked by looking at 
$$\omega_0^3 = (f^{*} \omega_X)^3 + 3(f^{*} \omega_X)^2 \wedge  \ast \left[f^{*} (\omega_X)\right]+ 3(f^{*} \omega_X) \wedge  \ast \left[f^{*} (\omega_X)^2 \right] + \ast \left[f^{*} (\omega_X \right]^3. $$
Since $(f^{*} \omega_X)^3 = 0$ and $\left(\ast \left[f^{*} \omega_X\right]\right)^2 = 0$, this 6-form reduces to $\omega_0^3 = 3 \beta \wedge \ast \beta $ with $\beta = f^{*} \omega_X^2$, which is positive outside the singularity set. To transition to the description of the 2-forms in coordinates, first we notice that all our near-symplectic forms can be expressed in the following way:
\begin{equation*}
\omega_0=\omega_1+f\omega_2+g\omega_3+h\omega_4
\end{equation*}
where $\omega_i \in \Omega^2(M)$, for $i=1, 2, 3, 4$, and $f,g,h \in C^{\infty}(M)$  are determined by each singularity. In all cases we have $\omega_1=du\wedge ds$, and up to an odd permutation and a minus sign, 
\begin{equation*}
\omega_2=dt\wedge dx+dy\wedge dz, \quad\omega_3=dt\wedge dy+dz\wedge dx,\quad \omega_4=dt\wedge dz+dx\wedge dy.
\end{equation*}
 Note that $\omega_1^2=0$, $\omega_2\wedge\omega_3=0, \, \omega_2\wedge\omega_4=0$, and $\omega_3\wedge\omega_4=0 $. Thus, we have
\begin{equation*}
\omega_0^2=f^2\omega_2^2+g^2\omega_3^2+h^2\omega_4^2
\end{equation*}
and
\begin{equation*}
\omega_0^3=f^2\omega_1\wedge\omega_2^2+g^2\omega_1\wedge\omega_3^2+h^2\omega_1\wedge\omega_4^2+f^3\omega_2^3+g^3\omega_3^3+h^3\omega_4^3.
\end{equation*}
This implies
\begin{eqnarray*}
\omega_0^3&=&f^2\omega_1\wedge\omega_2^2+g^2\omega_1\wedge\omega_3^2+h^2\omega_1\wedge\omega_4^2\\
&=&(f^2\omega_1+g^2\omega_1+h^2\omega_1)\wedge(2dt\wedge dx\wedge dy\wedge dz)\\
&=&(f^2+g^2+h^2)du\wedge ds\wedge dt\wedge dx\wedge dy\wedge dz.
\end{eqnarray*}
This form is clearly positive outside the singularity set. On each singularity we have that $\omega_0^3 = \omega_0^2 = 0$, since $df_4 =0$ at $\textnormal{Crit}_f$.  At each critical point $p\in M$ we find a 4-dimensional kernel  $K_p = \lbrace v\in T_pM \, \mid \, \omega_p(v, \cdot) = 0 \rbrace$ spanned by $\langle \partial_t, \partial_x, \partial_y, \partial_z \rangle$, and the rank of $D_K \colon K_p \rightarrow \Lambda^2 K^*$ is three.  With these properties $\omega_0$ is near-positive, i.e. $\omega_0^3 \geq 0$, and it satisfies the transversality condition. The only condition we lack now for this form to be near-symplectic is for it to be closed. 
\\
\\
The 2-form \eqref{omega0-general} is closed only around the fold singularities but not for the other three. Thus, we need to add a suitable 2-form $\eta$ so that $\omega_A = \omega_0 + \eta$ is closed. Fix $g$ on $K$, such that $\omega|_{K}$ is self-dual. We can then define a rescaling map $R_{\varepsilon}\colon \Omega^2K^{*} \rightarrow \Omega^2K^{*}$ and apply it to $\omega_0$. Finally, we add a small $\varepsilon$ to preserve the non-degeneracy. 
\begin{equation}\label{eq:general-ns-form}
\omega =  R_{\varepsilon}(\omega_0) + \varepsilon\cdot \eta.
\end{equation}
With a suitable choice of $\eta$, it can be checked that the near-positive properties of $\omega_0$ are preserved and it is closed. Thus, the 2-form \eqref{eq:general-ns-form} provides the desired near-symplectic form. 
\subsection{Folds}

Since this singularity is also present in generalized bLfs, the proof for this case follows exactly as in the proof of Theorem 1 in \cite{V16}. We will only recall a couple of useful facts that will be applied to the other singularities. The 4-dimensional kernel of the near-symplectic form is $K = \nu \oplus NZ$, where $\nu$ is the line bundle defined by $\nu = \ker(f^{*}\omega_X)$ and $NZ$ is the normal bundle of the singular locus $Z$. Using the same coordinates parametrizing the folds we can express the 4-dimensional tangent subspace as $K = \text{span}\langle \partial_t, \partial_x, \partial_y , \partial_z \rangle$. Let $f_4(u,s,t,x,y,z) = \frac{1}{2}(x^2 + y^2) - z^2$. The local model described in step 1 of \cite{V16} can be expressed as:  
\begin{align}
\omega =&  f^{*} \omega_X +  \ast \left[ (f^{*} (\omega_X^2) \right]
\nonumber 
\\
=&  du\wedge ds + x dt\wedge dx + y dt \wedge dy - 2 z dt \wedge dz 
\nonumber
\\
&+\ast ( x \, du\wedge ds \wedge dt \wedge dx + y\,  du\wedge ds \wedge dt \wedge dy - 2 z \, du \wedge ds \wedge dt \wedge dz)
\nonumber
\\
   =& du \wedge ds + x(dt \wedge dx + dy\wedge dz) + y (dt\wedge dy + dz\wedge dx) - 2z (dt \wedge dz + dx \wedge dy) 
\end{align}
This 2-form is already near-symplectic so it does not require any rescaling nor additional terms and $\eta = 0$.  
\subsection{Cusps}

A generalized wrinkled fibration has real and oriented coordinate charts around cusps with parametrization given by 
$$ f\colon (u,s,t, x,y,z) \mapsto (u,s,t,  x^3  - 3 t \cdot x + y^2 -  z^2). $$
Following Lekili's scheme \cite{L09}, we start with the 2-form
\begin{align}
\omega_0 =&  f^{*} \omega_X +  \ast \left[ (f^{*} (\omega_X^2) \right]
\nonumber 
\\
             = & du\wedge ds + 3(x^2-t)(dt\wedge dx + dy\wedge dz) + 2y(dt\wedge dy -dx\wedge dz) \\
             &  -2z (dt\wedge dz + dx\wedge dy). 
\nonumber
\end{align}

This 2-form $\omega_0$ is near-positive,  the kernel $K_p = \ker(\omega(p)) \subset T_p M$ is 4-dimensional spanned by $\langle \partial_t, \partial_x, \partial_y , \partial_z \rangle$, and the rank of $D_K$ at the singular points is 3.  This form is not closed though, as $d\omega_0 = 6x dx\wedge dy\wedge dz -3dt\wedge dy \wedge dz$. We modify it, and add the 2-form  $\eta =  - 6xy (dz\wedge dy) -3y (dt\wedge dx)$. To preserve the positivity on the fibres we introduce a scaling map. Locally we have splitting $T_pM = K_p \oplus \textnormal{Symp}_{Z_{p}}$. Equipping $M$ with a Riemannian metric $g$, we can restrict $g_K:= g|_{K}$ such that $\omega|_{K}$ is self-dual, and consider the Hodge-$\ast$ operator $\ast_{g_{K}}\colon \Omega^2 K^{*} \rightarrow \Omega^2 K^{*} $. Thus we can define a scaling map $R_{\varepsilon}\colon \Omega^{2}_{+}K^{*} \to \Omega^{2}_{+} K^{*}$  on basis elements of the space of self-dual forms on $K$:
\begin{align*}
 R_{\varepsilon}(dt\wedge dx + dy \wedge dz ) &= \varepsilon (dt \wedge dx + dy\wedge dx)  \\ 
 R_{\varepsilon}(dt \wedge dy + dz\wedge dx) &= dt \wedge dy + dz\wedge dx \\
 R_{\varepsilon}(dt\wedge dz + dx \wedge dy)  &= dt\wedge dz + dx \wedge dy  
\end{align*}
Applying $R_{\varepsilon}$, we now find the near symplectic form $\omega$ adapted to a neighborhood of the cusp singularity, as intended
$$\omega =    R_{\varepsilon}(\omega_0) + \varepsilon\cdot \eta =  du \wedge ds + R_{\varepsilon} (dt \wedge df_4 + \ast_{g_K}(dt \wedge df_4)) +  \varepsilon\cdot \eta. $$
Expanding the previous expression in coordinates we obtain 
\begin{eqnarray}\label{ns-form-cusp}
\omega  &=& du\wedge ds + 3\varepsilon(x^2-t)(dt\wedge dx + dy\wedge dz) + 2y dt\wedge dy + (2y - 6\varepsilon xy ) dz \wedge dx 
\nonumber
\\
& &- (2z + 3\varepsilon y) dt\wedge dz - 2z dx\wedge dy.
\end{eqnarray}
A basis for the tangent space of the fibre is given by the vectors: 
\begin{equation*}
v_1 =  \left(\frac{2z}{3(x^2-t)} \right) \frac{\partial}{\partial x} +  \frac{\partial}{\partial z},
\end{equation*}
\begin{equation*} v_2 =  \left(\frac{2y}{3(x^2-t)}\right) \frac{\partial}{\partial x} -  \frac{\partial}{\partial y}.
\end{equation*} By evaluating the 2-form $\omega$  on tangent vectors to the fibres we can see that for a sufficiently small $\varepsilon$ the form \eqref{ns-form-cusp} is positive on the fibres of $f$ (see appendix B).

\subsection{Swallowtails}
The coordinate charts around a swallowtail are given by 
$$f\colon (u,s,t, x,y,z) \mapsto ( u,s,t,  x^4 +  sx^2 + tx+ y ^2  -  z^2).$$
Define our initial form $\omega_0$ to be 
\begin{align*}
\omega_0 =&  f^{*} \omega_X + \ast \left[ (f^{*} (\omega_X^2) \right]
\\
             = & du\wedge ds + (4x^3 + 2sx + t) (dt \wedge dx + dx\wedge dy)+ 2y(dt\wedge dy -dx\wedge dz) \\
             &  -2z (dt\wedge dz + dx\wedge dy).              
\end{align*}

This form is non-degenerate outside the critical set and evaluates positively on the fibres of $f$.  At a critical point $p\in M$ we have a splitting $T_pM = K_p \oplus \textnormal{Symp}_Z$, where $K = \textnormal{span}\langle \partial_t, \partial_x, \partial_y, \partial_z \rangle$ and $\textnormal{Symp}_Z \subset TZ$ is the symplectic subspace given by $du\wedge ds$. 

However, this 2-form is not closed. Thus, we add the following extra terms to $\omega_0$,
\begin{eqnarray*}
\eta &=& -2z dt\wedge dy + (12x^2 -2s)y dz\wedge dx -y dt\wedge dz -(12 x^2 -2s) 2z dx\wedge dy 
\\
& &- x^2 dt \wedge ds - 2yz \, ds\wedge dx + 2xz \, ds \wedge dy
\end{eqnarray*}
and obtain  $\omega = \omega_0 + \eta$.  This 2-form is now closed. To preserve the non-degeneracy, we multiply $\omega_0$ by the function $R_{\varepsilon}$ and the 2-form $\eta$ by $\varepsilon$. 
\begin{eqnarray}\label{swalllowtail:form}
& & 
\\
\omega  &=& R_{\varepsilon} (\omega_0) +  \varepsilon \cdot \eta 
\nonumber
\\
&=& du\wedge ds +  \varepsilon(4x^3 + 2sx + t) (dt \wedge dx + dx\wedge dy)+ 2y(dt\wedge dy -dx\wedge dz) 
-2z (dt\wedge dz + dx\wedge dy)   
\nonumber 
\\
& &+ \varepsilon[-2z dt\wedge dy + (12x^2 -2s)y dz\wedge dx -y dt\wedge dz -(12 x^2 -2s) 2z dx\wedge dy 
\nonumber
\\
& &- x^2 dt \wedge ds - 2yz \, ds\wedge dx + 2xz \, ds \wedge dy]
\nonumber
\end{eqnarray}
This 2-form is closed, non-degenerate outside the singularity, and at the singular points it has a 4-dimensional kernel and $\textnormal{Rank}(D_K) = 3$. Thus, this is a near-symplectic form defined on a small neighbourhood around a swallowtail point.  

Using the basis of vectors tangent to the fibre given by
\begin{equation*}
 v_1 =  \left(\frac{2z}{4x^3 + 2sx +t} \right) \frac{\partial}{\partial x} +  \frac{\partial}{\partial z},
 \end{equation*}
 \begin{equation*}
  v_2 =  \left(\frac{2y}{4x^3 + 2sx +t)}\right) \frac{\partial}{\partial x} -  \frac{\partial}{\partial y},
  \end{equation*}  we can see that for a sufficiently small $\varepsilon$ the previous 2-form $\omega$ is positive on the fibres of $f$ (see appendix B).

\subsection{Butterflies}
The local model of generalized wrinkled fibration around a butterfly point is 
$$f\colon (u, s, t, x, y, z) \mapsto ( u,s,t,   x^5 + u x^3 + s x^2 + tx + y ^2  -  z^2).$$
Following the same scheme as for the other singularities, we begin with the 2-form 
\begin{align*}
\omega_0 =&  f^{*} \omega_X +   \ast \left[ (f^{*} (\omega_X^2) \right]
\\
             = & du\wedge ds + (5x^4 - 3ux^2 + 2sx - t) (dt \wedge dx + dx\wedge dy)+ 2y(dt\wedge dy -dx\wedge dz) \\
             &  -2z (dt\wedge dz +dx\wedge dy) - x^3 dt \wedge du + x^2 dt \wedge ds.
\end{align*}

This form has a 4-dimensional kernel $K$ at the singular points and the rank of $D_K$ is 3. It is also non-degenerate outside the critical set and evaluates positively on the fibres of $f$, but it is not closed.  By adding 
\begin{eqnarray*}
\eta &=&  (-10 x^3 +3 ux -s ) (4z dx \wedge dy + 2y dx \wedge dz) + 3x^2 ( 2 du \wedge dz - y du\wedge dz ) 
\\
& & + y \, dt \wedge dz + 2z  \, dt \wedge dy - 2xy \, ds \wedge dz - 4xz \, ds \wedge dy
\end{eqnarray*}
then the 2-form $ \omega_0 + \eta$ is closed.  To preserve the non-degeneracy, we add a scaling factor and obtain the local near-symplectic form around the butterfly point $\omega = R_{\varepsilon} (\omega_0) +  \varepsilon \cdot \eta$. In coordinates this is %
\begin{align}\label{butterfly:form}
\omega =& du\wedge ds + \varepsilon \cdot (5x^4 - 3ux^2 + 2sx - t) (dt \wedge dx + dx\wedge dy)+ 2y(dt\wedge dy -dx\wedge dz) 
\nonumber
\\
             &  -2z (dt\wedge dz + dx\wedge dy) - \varepsilon \cdot x^3 dt \wedge du + \varepsilon \cdot x^2 dt \wedge ds
\nonumber
\\
&+\varepsilon[(-10 x^3 +3 ux -s ) (4z dx \wedge dy + 2y dx \wedge dz) + 3x^2 ( 2 du \wedge dz - y du\wedge dz ) 
\nonumber
\\
&+ y \, dt \wedge dz + 2z  \, dt \wedge dy - 2xy \, ds \wedge dz - 4xz \, ds \wedge dy].
\nonumber
\end{align}

\noindent A basis of tangent vectors to the fibres is given by 
\begin{equation*}
v_1 = \left( 0, 0, 0, \frac{2z}{5x^4 -3ux^2 +2sx -t}, 0, 1\right)\quad {\rm and}
 \end{equation*}
 \begin{equation*}
 v_2 = \left( 0, 0 ,0 , \frac{2y}{5x^4 -3ux^2 +2sx -t} , -1, 0 \right).
 \end{equation*}  For a sufficiently small $\varepsilon$, the 2-form $\omega$ is positive on the fibres of $f$ (see appendix B). 

\end{proof}

\appendix
\section{Computations of local expressions }
Here we present the details concerning the calculations obtained in Corollaries \ref{Cor:localPoisson}, \ref{Cor:LocalSymplectic1}, and \ref{Cor:LocalSymplectic2}. This might allow for an easier verification of the results.

\subsection{Local expressions for the Poisson structures}\label{Appendix:Poisson}

\noindent {\bf Local expressions near a fold singularity}

The local coordinate model around a fold singularity is given by the map:
\begin{equation*}
(t_1, t_2, t_3, x_1, x_2, x_3) \mapsto (t_1, t_2, t_3, -x_1^2 + x_2^2 + x_3^2)
\end{equation*}

Considering each coordinate function as a Casimir function for the Poisson bivector that we want to find, we compute the differential matrix of the map
\begin{equation*}\left(
\begin{array}{cccc}
 1 & 0 & 0 & 0 \\
 0 & 1 & 0 & 0 \\
 0 & 0 & 1 & 0 \\
 0 & 0 & 0 & -2 x_1 \\
 0 & 0 & 0 & 2 x_2 \\
 0 & 0 & 0 & 2 x_3 \\
\end{array}
\right).
\end{equation*}

As we described, this gives a bivector matrix, which in this case is:

\begin{equation*}
\left(
\begin{array}{cccccc}
 0 & 0 & 0 & 0 & 0 & 0 \\
 0 & 0 & 0 & 0 & 0 & 0 \\
 0 & 0 & 0 & 0 & 0 & 0 \\
 0 & 0 & 0 & 0 & 2k x_3& -2 k x_2 \\
 0 & 0 & 0 & -2 k x_3 & 0 & -2 k x_1 \\
 0 & 0 & 0 & 2 k x_2 & 2 k x_1 & 0 \\
\end{array}
\right)
\end{equation*}

Therefore the Poisson structure in the local coordinates of a fold singularity is described by equation \ref{biv:fold}.
\medskip

We also compute the Poisson bivector for definite singularities for each wrinkled fibration. In this case, they are locally modeled by (\ref{eqn:fold-def1}) and (\ref{eqn:fold-def2}):
\begin{equation*}(t_1, t_2, t_3, x_1, x_2, x_3)\mapsto (t_1, t_2, t_3, x_1^2+x_2^2+x_3^2)
\end{equation*}

Following the same computations as above, the Poisson matrix is then:
\begin{equation*}
\left(
\begin{array}{cccccc}
 0 & 0 & 0 & 0 & 0 & 0 \\
 0 & 0 & 0 & 0 & 0 & 0 \\
 0 & 0 & 0 & 0 & 0 & 0 \\
 0 & 0 & 0 & 0 & 2 x_3 & -2 x_2 \\
 0 & 0 & 0 & -2 x_3 & 0 & 2 x_1\\
 0 & 0 & 0 & 2 x_2 & -2 x_1 & 0 \\
\end{array}
\right)
\end{equation*}

It follows that the Poisson bivector is given by \ref{biv:fold-def1}.
\medskip

For the case when the map is
\begin{equation*}(t_1, t_2, t_3, x_1, x_2, x_3)\mapsto (t_1, t_2, t_3, x_1^2+x_2^2+x_3^2).
\end{equation*}

the Poisson matrix is then:
\begin{equation*}
\left(
\begin{array}{cccccc}
 0 & 0 & 0 & 0 & 0 & 0 \\
 0 & 0 & 0 & 0 & 0 & 0 \\
 0 & 0 & 0 & 0 & 0 & 0 \\
 0 & 0 & 0 & 0 & -2 x_3 & 2 x_2 \\
 0 & 0 & 0 & 2 x_3 & 0 & -2 x_1\\
 0 & 0 & 0 & -2 x_2 & 2 x_1 & 0 \\
\end{array}
\right)
\end{equation*}

Hence the Poisson bivector is given by \ref{biv:fold-def2}.
\newline

\noindent{\bf Local expressions near a cusp singularity.}

The local coordinate model around a cusp singularity is given by:
\begin{equation*}
(t_1, t_2, t_3, x_1, x_2, x_3) \mapsto (t_1, t_2, t_3, x_1^3 - 3 t_1x_1 + x_2^2 - x_3^2)
\end{equation*}

The differential matrix of the map is
\begin{equation*}
\left(
\begin{array}{cccc}
 1 & 0 & 0 & -3 x_1 \\
 0 & 1 & 0 & 0 \\
 0 & 0 & 1 & 0 \\
 0 & 0 & 0 & 3 x_1^2-3 t_1 \\
 0 & 0 & 0 & 2 x_2 \\
 0 & 0 & 0 & -2 x_3 \\
\end{array}
\right)
\end{equation*}

The corresponding bivector matrix is
\begin{equation*}
\left(
\begin{array}{cccccc}
 0 & 0 & 0 & 0 & 0 & 0 \\
 0 & 0 & 0 & 0 & 0 & 0 \\
 0 & 0 & 0 & 0 & 0 & 0 \\
 0 & 0 & 0 & 0 & -2 k x_3 & -2k x_2 \\
 0 & 0 & 0 & 2 k x_3 & 0 & 3 k(x_1^2- t_1) \\
 0 & 0 & 0 & 2 k x_2 & 3 k (t_1-x_1^2) & 0 \\
\end{array}
\right).
\end{equation*}

Thus, the Poisson bivector in the local coordinates of a cusp singularity is given by \ref{biv:cusp}.
\newline

For definite singularities in cusps, we obtain in each case (\ref{eqn:cusp-def1}) and (\ref{eqn:cusp-def2}):
\begin{equation*}(t_1, t_2, t_3, x_1, x_2, x_3)\mapsto (t_1, t_2, t_3, x_1^3-3t_1x_1 +x_2^2+ x_3^2)
\end{equation*}

The Poisson matrix is:
\begin{equation*}
\left(
\begin{array}{cccccc}
 0 & 0 & 0 & 0 & 0 & 0 \\
 0 & 0 & 0 & 0 & 0 & 0 \\
 0 & 0 & 0 & 0 & 0 & 0 \\
 0 & 0 & 0 & 0 & 2 x_3 & -2 x_2 \\
 0 & 0 & 0 & -2 x_3 & 0 & 3 x_1^2-3 t_1 \\
 0 & 0 & 0 & 2 x_2 & 3 t_1-3 x_1^2 & 0 \\
\end{array}
\right).
\end{equation*}

Then the corresponding bivector is \ref{biv:cusp-def1}.
\medskip

For the case:
\begin{equation*}(t_1, t_2, t_3, x_1, x_2, x_3)\mapsto (t_1, t_2, t_3, x_1^3-3t_1x_1 -x_2^2- x_3^2)
\end{equation*}

The Poisson matrix is:
\begin{equation*}
\left(
\begin{array}{cccccc}
 0 & 0 & 0 & 0 & 0 & 0 \\
 0 & 0 & 0 & 0 & 0 & 0 \\
 0 & 0 & 0 & 0 & 0 & 0 \\
 0 & 0 & 0 & 0 & -2 x_3 & 2 x_2 \\
 0 & 0 & 0 & -2 x_3 & 0 & 3 x_1^2-3 t_1 \\
 0 & 0 & 0 & -2 x_2 & 3 t_1-3 x_1^2 & 0 \\
\end{array}
\right)
\end{equation*}

Therefore the Poisson bivector is given by the equation \ref{biv:cusp-def2}.
\newline

\noindent{\bf Local expressions near a swallowtail singularity}

The local coordinate model around a swallowtail singularity is given by the map:
\begin{equation*}
(t_1, t_2, t_3, x_1, x_2, x_3) \mapsto (t_1, t_2, t_3, x_1^4 + t_1 x_1^2 + t_2 x_1 + x_2^2 - x_3^2)
\end{equation*}

Its differential matrix is:
\begin{equation*}
\left(
\begin{array}{cccc}
 1 & 0 & 0 & x_1^2 \\
 0 & 1 & 0 & x_1 \\
 0 & 0 & 1 & 0 \\
 0 & 0 & 0 & 4 x_1^3+2 t_1 x_1+t_2\\
 0 & 0 & 0 & 2 x_2 \\
 0 & 0 & 0 & -2 x_3 \\
\end{array}
\right)
\end{equation*}

The corresponding matrix is:
\begin{equation*}
\left(
\begin{array}{cccccc}
 0 & 0 & 0 & 0 & 0 & 0 \\
 0 & 0 & 0 & 0 & 0 & 0 \\
 0 & 0 & 0 & 0 & 0 & 0 \\
 0 & 0 & 0 & 0 & -2k x_3 & -2 k x_2 \\
 0 & 0 & 0 & 2 k x_3 & 0 & k(4 x_1^3+2 t_1 x_1+t_2) \\
 0 & 0 & 0 & 2 k x_2 & k(-4 x_1^3-2 t_1 x_1-t_2) & 0 \\
\end{array}
\right)
\end{equation*}

It produces the Poisson bivector in the local coordinates of a swallowtail singularity described by equation \ref{biv:swallowtail}.
\newline

For the corresponding definite singularities:
\begin{equation*}(t_1, t_2, t_3, x_1, x_2, x_3)\mapsto (t_1, t_2, t_3, x_1^4+t_1x_1^2+t_2x_1+x_2^2+x_3^2)
\end{equation*}

The Poisson matrix is:
\begin{equation*}
\left(
\begin{array}{cccccc}
 0 & 0 & 0 & 0 & 0 & 0 \\
 0 & 0 & 0 & 0 & 0 & 0 \\
 0 & 0 & 0 & 0 & 0 & 0 \\
 0 & 0 & 0 & 0 & 2 x_3 & -2 x_2 \\
 0 & 0 & 0 & -2 x_3 & 0 & 4 x_1^3+2 t_1 x_1+t_2 \\
 0 & 0 & 0 & 2 x_2 & -4 x_1^3-2 t_1x_1-t_2 & 0 \\
\end{array}
\right)
\end{equation*}

The Poisson bivector is \ref{biv:swallowtail-def1}.
\medskip

In the case when the local form of the definite singularity is:
\begin{equation*}(t_1, t_2, t_3, x_1, x_2, x_3)\mapsto (t_1, t_2, t_3, x_1^4+t_1x_1^2+t_2x_1-x_2^2-x_3^2)
\end{equation*}

The Poisson matrix is:
\begin{equation*}
\left(
\begin{array}{cccccc}
 0 & 0 & 0 & 0 & 0 & 0 \\
 0 & 0 & 0 & 0 & 0 & 0 \\
 0 & 0 & 0 & 0 & 0 & 0 \\
 0 & 0 & 0 & 0 & -2 x_3 & 2 x_2 \\
 0 & 0 & 0 & 2 x_3 & 0 & 4 x_1^3+2 t_1 x_1+t_2 \\
 0 & 0 & 0 & -2 x_2 & -4 x_1^3-2 t_1x_1-t_2 & 0 \\
\end{array}
\right)
\end{equation*}

The corresponding bivector is \ref{biv:swallowtail-def2}.
\newline

\noindent{\bf Local expressions near a butterfly singularity}

The local coordinate model around a buttterfly singularity is given by:
\begin{equation*}
(t_1, t_2, t_3, x_1, x_2, x_3) \mapsto (t_1, t_2, t_3, x_1^5 + t_1 x_1^3 + t_2 x_1^2 + t_3 x_1 + x_2^2 - x_3^2)
\end{equation*}

The differential of the map is:
\begin{equation*}
\left(
\begin{array}{cccc}
 1 & 0 & 0 & x_1^3 \\
 0 & 1 & 0 & x_2^2 \\
 0 & 0 & 1 & x_1 \\
 0 & 0 & 0 & 5 x_1^4+3 t_1 x_1^2+2 t_2x_1+t_3 \\
 0 & 0 & 0 & 2 x_2 \\
 0 & 0 & 0 & -2 x_3 \\
\end{array}
\right)
\end{equation*}

The corresponding matrix is:
\begin{equation*}
\left(
\begin{array}{cccccc}
 0 & 0 & 0 & 0 & 0 & 0 \\
 0 & 0 & 0 & 0 & 0 & 0 \\
 0 & 0 & 0 & 0 & 0 & 0 \\
 0 & 0 & 0 & 0 & -2 k x_3 & -2 k x_2 \\
 0 & 0 & 0 & 2 k x_3 & 0 & k(5 x_1^4+3 t_1 x_1^2+2 t_2 x_1+t_3) \\
 0 & 0 & 0 & 2 k x_2 & k(-5 x_1^4-3 t_1 x_1^2-2 t_2x_1-t_3) & 0 \\
\end{array}
\right)
\end{equation*}

Then the Poisson bivector in the local coordinates of a butterfly singularity is described by \ref{biv:butterfly}.
\newline

For definite singularities:
\begin{equation*}
(t_1, t_2, t_3, x_1, x_2, x_3) \mapsto (t_1, t_2, t_3, x_1^5 + t_1 x_1^3 + t_2 x_1^2 + t_3 x_1 + x_2^2 + x_3^2)\end{equation*}

The Poisson matrix is:
 \begin{equation*}
\left(
\begin{array}{cccccc}
 0 & 0 & 0 & 0 & 0 & 0 \\
 0 & 0 & 0 & 0 & 0 & 0 \\
 0 & 0 & 0 & 0 & 0 & 0 \\
 0 & 0 & 0 & 0 & 2 x_3 & -2 x_2 \\
 0 & 0 & 0 & -2 x_3& 0 & 5 x_1^4+3 t_1 x_1^2+2 t_2 x_1+t_3 \\
 0 & 0 & 0 & 2 x_2 & -5 x_1^4-3 t_1 x_1^2-2 t_2 x_1-t_3 & 0 \\
\end{array}
\right)
\end{equation*}

Then the corresponding bivector is \ref{biv:butterfly-def1}.
\medskip

When the local form of the definite singularity is: 
\begin{equation*}(t_1, t_2, t_3, x_1, x_2, x_3) \mapsto (t_1, t_2, t_3, x_1^5 + t_1 x_1^3 + t_2 x_1^2 + t_3 x_1 - x_2^2 - x_3^2)
\end{equation*}

The Poisson matrix is:
\begin{equation*}
\left(
\begin{array}{cccccc}
 0 & 0 & 0 & 0 & 0 & 0 \\
 0 & 0 & 0 & 0 & 0 & 0 \\
 0 & 0 & 0 & 0 & 0 & 0 \\
 0 & 0 & 0 & 0 & -2 x_3 & 2 x_2 \\
 0 & 0 & 0 & 2 x_3& 0 & 5 x_1^4+3 t_1 x_1^2+2 t_2 x_1+t_3 \\
 0 & 0 & 0 & -2 x_2 & -5 x_1^4-3 t_1 x_1^2-2 t_2 x_1-t_3 & 0 \\
\end{array}
\right)
\end{equation*}
 
Then the bivector is \ref{biv:butterfly-def2}.
\newline
 
\subsection{Equations for the symplectic forms on the leaves near singularities}

\noindent{\bf Indefinite Fold}

As we described in the general proccedure, if $u_q, v_q$ are tangent vectors to the leaves there exist
co-vectors $\alpha_q, \beta_q\in T^*_qM$ such that $\mathcal{B}_q
(\alpha_q)=u_q$ and $\mathcal{B}_q (\beta_q)=v_q$, where the map
$\mathcal{B}_q$ is given by:
\begin{equation*}
\mathcal{B}_q(\alpha)(\cdot)=\pi_q(\cdot, \alpha)
\end{equation*}

Therefore, if we want to find two tangent vectors to the symplectic leaves we have to give vectors annihilated simultaneously by the differential of
four Casimir functions for the corresponding Poisson structure.

A straightforward calculation yields that the vectors,
\begin{equation*}
u_q=\frac{x_3\frac{\partial}{\partial x_1}+x_1\frac{\partial}{\partial x_3}}{(x_1^2 + x_3^2)^{1/2}}
\end{equation*}
\begin{equation*}
v_q=\frac{x_1^2x_2\frac{\partial}{\partial x_1}+x_1\frac{\partial}{\partial x_2}-x_1x_2x_3\frac{\partial}{\partial x_3}}{(x_1^2+x_3^2)^{1/2}}
\end{equation*}
are tangent to $\Sigma_q$ at $q$, and orthogonal with respect to
the euclidean metric $$ds^2=dt_1^2+dt_2^2+dt_3^2+dx_1^2+dx_2^2+dx_3^2$$ on $B^6$. Using the local expression of the Poisson structure for
a fold singularity given by equation (\ref{biv:fold}), one can check that $\mathcal{B}_q
(\alpha_q)=u_q$, for
\begin{equation*}
\alpha_q=\frac{x_3dx_1+x_1dx_2}{k(q)(x_1^2+x_3^2)^{1/2}}.
\end{equation*}

Similarly, $\mathcal{B}_q(\beta_q)=v_q$, for
\begin{equation*}
\beta_q= \frac{-x_1x_2x_3dx_2-x_1(x_1^2+x_3^2)}{2k(q) (x_1^2+x_3^2) }.
\end{equation*}

A direct calculation now implies that the symplectic form is given by \ref{FoldForm}:
\begin{equation*}
\omega_{\Sigma_q}(q)(u_q, v_q)=\langle \alpha_q,  v_q
\rangle=\frac{x_1^2}{2 k(q)(x_1^2+x_3^2)^{1/2}}
\end{equation*}
\newline

For definite singularities described by the equations (\ref{eqn:fold-def1}) and (\ref{eqn:fold-def2}) we obtain the symplectic forms
\begin{equation*}
\omega_{\Sigma_q}=-\frac{x_1^2}{2 (x_1^2+x_3^2)^{1/2}}\omega_{Area}(q)
\end{equation*}
and
\begin{equation*}
\omega_{\Sigma_q}=\frac{x_1^2}{2 (x_1^2+x_3^2)^{1/2}}\omega_{Area}(q)
\end{equation*}

respectively. This follows directly with the same computations of the previous case. Tangent vectors to the leaves $u_q$ and $v_q$ are slightly different, one component changes its sign. This creates a change of sign on one of the components of the corresponding vectors $\alpha_q$ and $\beta_q$.	
\newline

\noindent{\bf Indefinite Cusps}

In this case we find that the vectors,
\begin{equation*}
u_q=\frac{-2 x_3\frac{\partial}{\partial x_1}+3 \left(t_1-x_1^2\right)\frac{\partial}{\partial x_3}}{(9 \left(t_1-x_1^2\right)^2+4 x_3^2)^{1/2}}
\end{equation*}
\begin{equation*}
v_q=\frac{(2x_2-8x_2x_3^2)\frac{\partial}{\partial x_1}+3(t_1-x_1^2) (9(t_1 - x_1^2)^2 + 4 x_3^2)\frac{\partial}{\partial x_2}+12 (t_1 - x_1^2) x_2 x_3\frac{\partial}{\partial x_3}}{9 (t_1 - x_1^2)^2 + 4 x_3^2}
\end{equation*}
are tangent to $\Sigma_q$ at $q$, and orthogonal with respect to
the euclidean metric
$$ds^2=dt_1^2+dt_2^2+dt_3^2+dx_1^2+dx_2^2+dx_3^2$$

on $B^6$. Using the corresponding local expression of the bivector (\ref{biv:cusp}), we check that $\mathcal{B}_q(\alpha_q)=u_q$, for
\begin{equation*}
\alpha_q=\frac{3 (t_1 - x_1^2)dx_1+x_3dx_3}{2k(q)(9 t_1^2 - 18 t_1 x_1^2 + 9 x_1^4 + 4 x_3^2)^{1/2}}.
\end{equation*}

Similarly, $\mathcal{B}_q(\beta_q)=v_q$, for
\begin{equation*}
\beta_q= \frac{6 (t_1 - x_1^2) x_2 x_3dx_1-9 (t_1 - x_1^2)^2 x_2dx_3}{k(q) (9 (t_1 - x_1^2)^2 + 4 x_3^2) }.
\end{equation*}

Now a direct calculation gives that the symplectic form is \ref{CuspForm}.
\newline

For the definite singularities modelled by the equations (\ref{eqn:cusp-def1}) and (\ref{eqn:cusp-def2}), the corresponding symplectic forms on the leaves coincide with the previous one:
\begin{equation*}
\omega_{\Sigma_q}=\frac{3 (t_1 - x_1^2) x2}{(9 (t_1 - x_1^2)^2 + 4 x_3^2)^{1/2}}\omega_{Area}(q)
\end{equation*}

This last equality follows from very similar computations as in the previous case, up to a sign, as in the fold case.
\newline

\noindent{\bf Indefinite Swallowtail}

We find that the vectors,
\begin{equation*}
u_q=\frac{2x_3\frac{\partial}{\partial x_1}+(t_2 + 2 t_1 x_1 + 4 x_1^3)\frac{\partial}{\partial x_3}}{((t_2 + 2 t_1 x_1 + 4 x_1^3)^2 + 4 x_3^2)^{1/2}}
\end{equation*}
\begin{equation*}
v_q=\frac{(-2x_2(t_2 + 2 t_1 x_1 + 4 x_1^3)^2 + 4 x_3^2+8x_3^2)\frac{\partial}{\partial x_1}+(t_2 + 2 t_1 x_1 + 4 x_1^3)\frac{\partial}{\partial x_2}}{(t_2 + 2 t_1 x_1 + 4 x_1^3)^2 + 4 x_3^2}
\end{equation*}
\begin{equation*}
\quad \quad +\frac{4 (t_2 + 2 t_1 x_1 + 4 x_1^3) x_2 x_3\frac{\partial}{\partial x_3}}{(t_2 + 2 t_1 x_1 + 4 x_1^3)^2 + 4 x_3^2}
\end{equation*}
are tangent to $\Sigma_q$ at $q$, and orthogonal with respect to
the euclidean metric
$$ds^2=dt_1^2+dt_2^2+dt_3^2+dx_1^2+dx_2^2+dx_3^2$$
on $B^6$. Using the local expression of the Poisson structure for
a fold singularity given by equation (\ref{biv:fold}), one can check that $\mathcal{B}_q
(\alpha_q)=u_q$, for
\begin{equation*}
\alpha_q=\frac{(t_2 +2 t_1 x_1+4 x_1^3)dx_1-2x_3 dx_3}{2x_2k(q)((t_2 + 2 t_1 x_1 + 4 x_1^3)^2 + 4 x_3^2)^{1/2}}.
\end{equation*}

Similarly, $\mathcal{B}_q(\beta_q)=v_q$, for
\begin{equation*}
\beta_q= \frac{1}{k}\left(\frac{2 x_3(t_2 + 2 t_1 x_1 + 4 x_1^3) dx_1}{t_2^2 + 4 t_1 t_2 x_1 + 4 t_1^2 x_1^2 + 8 t_2 x_1^3 + 16 t_1 x_1^4 + 16 x_1^6 + 
 4 x_3^2}\right.
\end{equation*}
\begin{equation*}
\quad \quad +\left. \left(1-\frac{4x_3^2}{(t_2 + 2 t_1 x_1 + 4 x_1^3)^2 + 4 x_3^2}\right)dx_3\right)
\end{equation*}
A direct calculation now implies that the symplectic form is \ref{SwallowtailForm}.
\newline

For the definite singularities we obtain that the symplectic forms on the leaves coincide in both cases (\ref{eqn:swallowtail-def1}) and (\ref{eqn:swallowtail-def2}):
\begin{equation*}
\omega_{\Sigma_q}=-\frac{(t_2 + 2 t_1 x_1 + 4 x_1^3}{((t_2 + 2 t_1 x_1 + 4 x_1^3)^2 + 4 x_3^2)^{1/2}}\omega_{Area}(q)
\end{equation*}

Analogously to the last cases, we proceed changing the corresponding signs.
\newline

\noindent{\bf Butterfly singularity}

We find that the vectors,
\begin{equation*}
u_q=\frac{2x_3\frac{\partial}{\partial x_1}+(t_3 + x_1 (2 t_2 + 3 t_1 x_1 + 5 x_1^3))\frac{\partial}{\partial x_3}}{((t_3 + x_1 (2 t_2 + 3 t_1 x_1 + 5 x_1^3))^2 + 4 x_3^2)^{1/2}}
\end{equation*}
\begin{equation*}
v_q=\left(\frac{8x_2x_3^2}{(t_3 + x_1 (2 t_2 + 3 t_1 x_1 + 5 x_1^3))^2 + 4 x_3^2}-2x_2\right)\frac{\partial}{\partial x_1}
\end{equation*}
\begin{equation*}
\quad \quad +\left(t_3 + x_1 (2 t_2 + 3 t_1 x_1 + 5 x_1^3)\right)\frac{\partial}{\partial x_2}
\end{equation*}
\begin{equation*}
\quad \quad +\frac{4 (t_3 + x_1 (2 t_2 + 3 t_1 x_1 + 5 x_1^3)) x_2 x_3}{(t_3 + x_1 (2 t_2 + 3 t_1 \alpha_1 + 5 x_1^3))^2 + 4 x_3^2}\frac{\partial}{\partial x_3}
\end{equation*}
are tangent to $\Sigma_q$ at $q$, and orthogonal with respect to
the euclidean metric
$$ds^2=dt_1^2+dt_2^2+dt_3^2+dx_1^2+dx_2^2+dx_3^2$$
on $B^6$. Using the local expression of the Poisson structure for
a fold singularity given by equation (\ref{biv:fold}), one can check that $\mathcal{B}_q
(\alpha_q)=u_q$, for
\begin{equation*}
\alpha_q=\frac{(t_3 +2 t_2 x_1 +3 t_1 x_1^2 +5 x_1^4)dx_1-2x_3dx_3}{2x_2k(q)((t_3 + x_1 (2 t_2 + 3 t_1 x_1 + 5 x_1^3))^2 + 4 x_3^2)^{1/2}}
\end{equation*}

Similarly, $\mathcal{B}_q(\beta_q)=v_q$, for
\begin{equation*}
\beta_q= \frac{2x_3 (t_3 + x_1 (2 t_2 + 3 t_1 x_1 + 5 x_1^3)) }{(t_3 + x_1 (2 t_2 + 3 t_1 x_1 + 5 x_1^3))^2 + 4 x_3^2}dx_1
\end{equation*}
\begin{equation*}
\quad \quad +\left(1-\frac{4x_3^2}{(t_3 + x_1 (2 t_2 + 3 t_1 x_1 + 5 x_1^3))^2 + 4 x_3^2}\right)dx_3.
\end{equation*}

A direct calculation now implies that the symplectic form is given by \ref{ButterflyForm}:
\newline

We have that for the corresponding butterfly definite singularities the symplectic form is in both cases:
\begin{equation*}
\omega_{\Sigma_q}=-\frac{(t_3 + x_1 (2 t_2 + 3 t_1 x_1 + 5 x_1^3)}{(t_3 + 
    x_1 (2 t_2 + 3 t_1 x_1 + 5 x_1^3))^2 + 4 x_3^2)^{1/2}}\omega_{Area}(q)
\end{equation*}
\medskip

\section{Positivity on fibres of local near-symplectic forms}   

\noindent {\bf Cusp:}
\\
\noindent Near-symplectic form:
\begin{align}
\omega  &= du\wedge ds + 3\varepsilon(x^2-t)(dt\wedge dx + dy\wedge dz) + 2y dt\wedge dy + (2y - 6\varepsilon xy ) dz \wedge dx 
\nonumber
\\
&- (2z + 3\varepsilon y) dt\wedge dz - 2z dx\wedge dy.
\nonumber
\end{align}
\noindent Map:
$$ f\colon (u,s,t, x,y,z) \mapsto (u,s,t,  x^3  - 3 t \cdot x + y^2 -  z^2) $$
\noindent Differential:
$$
Df = \left(\begin{array}{cccccc}
	 1		         & 0               &   0             &       0        & 0            &       0                                                         \\
	      0             &  1		    &   0             &       0        & 0        &          0                                                      \\
	      0             &  0              & 1		       &       0        &  0           &         0                                                        \\
	      0             &   0             &    x            &      3(x^2 - t)       & 2y            &  -2z                                                               \\
\end{array}\right) 
$$
Basis of vectors tangent to the fibre:
\begin{align}
v_1 &= \left( 0, 0, 0, \frac{2z}{3(x^2 - t) }, 0, 1\right)
\nonumber
\\
v_2 &= \left( 0, 0 ,0 , \frac{2y}{3(x^2 - t) } , -1, 0 \right)
\nonumber
\end{align}

\noindent Relevant terms of the 2-form when evaluated on $v_1$ and $v_2$ :
$$\tilde{\omega} = 3\varepsilon(x^2-t)dy\wedge dz + (2y - 6\varepsilon xy ) dz \wedge dx - 2z dx\wedge dy.$$

\noindent Evaluating $\omega(v_1, v_2)$, which amounts to evaluate $\tilde{\omega}(v_1, v_2)$ we obtain:
\begin{align}
\omega(v_1, v_2)  =& \tilde{\omega}(v_1, v_2) = \frac{1}{3(x^2 -t)} \left( 3\varepsilon (x^2 -t)^2 + 4y^2 - 12 \varepsilon x y^2 + 4z^2 \right)
\nonumber
\\
=& \frac{1}{3(x^2 -t)} \left( 3\varepsilon (x^2 - t)^2 + 4y^2(1 - \varepsilon 3 x) + 4z^2 \right)
\end{align}

All terms are always positive except possibly $12\varepsilon xy^2$.  However, taking a sufficiently small neighbourhood with $|x|<1$, this term is at most $12\varepsilon y^2$.  By choosing $\varepsilon$ to be sufficiently small we can arrange that this term will be smaller that $4y^2$, hence the whole expression remains positive. 

\noindent {\bf Swallowtails:}
\\
\noindent Near-symplectic form:
\begin{align}
 \omega = & du\wedge ds + \varepsilon(4x^3 + 2sx +t) (dt \wedge dx + dy \wedge dz) + 2y (dt \wedge dy - dx \wedge dz) - 2z (dt \wedge dz + dx \wedge dy)
\nonumber
\\
&+ \varepsilon [ -2z dt\wedge dy + (12x^2 -2s) y dz\wedge dx - y dt \wedge dz - (12x^2 -2s)2z dx\wedge dy  
\nonumber
\\
&- x^2 dt \wedge ds -2yz ds \wedge dx + 2xz ds \wedge dy ]
\nonumber
\end{align}

\noindent Map:
$$f\colon (u,s,t,x,y,z) \mapsto (u, s, t, x^4 + sx^2 - tx + y^2 - z^2) $$

\noindent Differential:
$$
Df = \left(\begin{array}{cccccc}
	 1		         & 0               &   0             &       0        & 0            &       0                                                         \\
	      0             &  1		    &   0             &       0        & 0        &          0                                                      \\
	      0             &  0              & 1		       &       0        &  0           &         0                                                        \\
	      0             &   x^2             &    x            &      4x^3 + 2sx -t        & 2y            &  -2z                                                               \\
\end{array}\right) 
$$
Basis of vectors tangent to the fibre:
\begin{align}
v_1 &= \left( 0, 0, 0, \frac{2z}{4x^3 + 2xs -t}, 0, 1\right)
\nonumber
\\
v_2 &= \left( 0, 0 ,0 , \frac{2y}{4x^3 + 2xs -t} , -1, 0 \right)
\nonumber
\end{align}

\noindent Relevant terms of the 2-form when evaluated on $v_1$ and $v_2$ 
\begin{align}
\tilde{\omega} &=  \varepsilon(4x^3 + 2sx -t) (dt \wedge dx + dy \wedge dz) + 2y (dt \wedge dy - dx \wedge dz) - 2z (dt \wedge dz + dx \wedge dy)
\nonumber
\\
&+ \varepsilon [ (12x^2 -2s) y dz\wedge dx - (12x^2 -2s)2z dx\wedge dy  ]
\nonumber
\\
\nonumber
\\
&=  \varepsilon(4x^3 + 2sx -t) (dt \wedge dx + dy \wedge dz) + y[\varepsilon(12x^2 - 2s) + 2] (dt \wedge dy - dx \wedge dz) 
\nonumber
\\
&- 2z[\varepsilon(12x^2 -2s) + 1] (dt \wedge dz + dx \wedge dy)
\nonumber
\end{align}
\\
Evaluating $\omega(v_1, v_2)$, which amounts to evaluate $\tilde{\omega}(v_1, v_2)$, we obtain:

\begin{align}
\omega(v_1, v_2)  =& \tilde{\omega}(v_1, v_2) = 
\nonumber
\\
=& \frac{1}{4x^3 + 2sx - t} \left( \varepsilon (4x^3 + 2sx - t)^2 + 2y^2 [ \varepsilon(12x^2 - 2s) +2) + 4z^2(\varepsilon(12x^2 - 2s) + 1) \right)
\end{align}

By restricting $s$ to a sufficiently small interval around $0$ which can be scaled by $\varepsilon$, the previous expression remains positive. 

\noindent {\bf Butterfly:}
\\
\noindent Near-symplectic form:
\begin{align}
 \omega = & du\wedge ds + \varepsilon(5x^4 - 3ux^2 + 2sx - t) (dt \wedge dx + dy \wedge dz) + 2y (dt \wedge dy - dx \wedge dz) 
\nonumber
\\
&- 2z (dt \wedge dz + dx \wedge dy) - \varepsilon dt\wedge du + \varepsilon x^2 dt \wedge ds 
\nonumber
\\
&+\varepsilon[(-10 x^3 +3 ux -s ) (4z dx \wedge dy + 2y dx \wedge dz) + 3x^2 ( 2 du \wedge dz - y du\wedge dz ) 
\nonumber
\\
&+ y \, dt \wedge dz + 2z  \, dt \wedge dy - 2xy \, ds \wedge dz - 4xz \, ds \wedge dy]
\nonumber
\end{align}

\noindent Map:
$$f\colon (u,s,t,x,y,z) \mapsto (u, s, t, x^5 -ux^3 + sx^2 - tx + y^2 - z^2) $$

\noindent Differential:
$$
Df = \left(\begin{array}{cccccc}
	 1		         & 0               &   0             &       0        & 0            &       0                                                         \\
	      0             &  1		    &   0             &       0        & 0        &          0                                                      \\
	      0             &  0              & 1		       &       0        &  0           &         0                                                        \\
	      -x^3             &   x^2             &   - x            &    5x^4 -3ux^2 +2sx -t        & 2y            &  -2z                                                               \\
\end{array}\right) 
$$
Basis of vectors tangent to the fibre:
\begin{align}
v_1 &= \left( 0, 0, 0, \frac{2z}{5x^4 -3ux^2 +2sx -t}, 0, 1\right)
\nonumber
\\
v_2 &= \left( 0, 0 ,0 , \frac{2y}{5x^4 -3ux^2 +2sx -t} , -1, 0 \right)
\nonumber
\end{align}

\noindent Relevant terms of the 2-form when evaluated on $v_1$ and $v_2$ :
\begin{align}
\tilde{\omega} = \, &  \varepsilon(5x^4 - 3ux^2 + 2sx - t) (dt \wedge dx + dy \wedge dz) + 2y (dt \wedge dy - dx \wedge dz)
\nonumber\\
& \, - 2z (dt \wedge dz + dx \wedge dy) + \varepsilon[(-10 x^3 +3 ux -s ) (4z dx \wedge dy + 2y dx \wedge dz)]
\end{align}
\\
Let $W(x,t)=5x^4 - 3ux^2 + 2sx - t$. Evaluating $\omega(v_1, v_2)$, we obtain:
\begin{align}
\omega(v_1, v_2)  =& \tilde{\omega}(v_1, v_2) = 
\nonumber
\\
=& \frac{1}{W(x,t)} \left( \varepsilon (5x^4 - 3ux^2 + 2sx - t)^2 + 4y^2 + 4z^2 + \varepsilon(-10x^3 - 3ux -s)(4y^2 + 4z^2)   \right)
\nonumber
\\
=& \frac{1}{W(x,t)} [ \varepsilon (5x^4 - 3ux^2 + 2sx - t)^2 + 4y^2 (1+ \varepsilon(-10x^3 - 3ux -s))  
\nonumber
\\
&+ 4z^2 (1+ \varepsilon(-10x^3 - 3ux -s)) ]
\end{align}

All terms are always positive except possibly $\varepsilon(-10x^3 - 3ux -s)$. By restricting $u$ and $s$ around $0$ to a sufficiently small neighbourhood with $|x|<1$, and by choosing a sufficiently small $\varepsilon$, we can bound this term so that it is smaller than 1, so that the whole expression remains positive. 


\providecommand{\bysame}{\leavevmode\hbox to3em{\hrulefill}\thinspace}


\begin{thebibliography}{99}

\bibitem{ADK05} D. Auroux, S. K. Donaldson, L. Katzarkov, {\em  Singular Lefschetz pencils}, Geometry \& Topology Vol. 9 (2005) 1043 --1114.




 \bibitem{BW05} H. Bursztyn, A. Weinstein, {\em Poisson geometry and Morita equivalence,} Poisson geometry, deformation quantisation and group representations, 1--78, London Math. Soc. Lecture Note Ser., 323, Cambridge Univ. Press, Cambridge, 2005. 


\bibitem{CF03} M. Crainic, R.L. Fernandes, {\it
Integrability of Lie brackets,}
Ann. of Math. (2) 157 (2003), no. 2, 575--620.



\bibitem{CF04} \bysame , {\it
Integrability of Poisson brackets,}
J. Differential Geom. 66 (2004), no. 1, 71--137.

\bibitem{CK16} G. Cavalcanti, R. Klaasse, {\it Fibrations and log-symplectic structures},  ArXiv e-prints,  arXiv:1606.00156, (2016), 23 pp. 

\bibitem{Dam89} P. A. Damianou, {\em Nonlinear Poisson Brackets}. Ph.D. Dissertation, University of Arizona, 1989.

\bibitem{DamP12} P. A. Damianou, F. Petalidou, {\em Poisson Brackets with Prescribed Casimirs}. Canadian J. Math. Vol. 64 (5), (2012),
991--1018.

\bibitem{D99} S. K. Donaldson, {\em Lefschetz pencils on symplectic manifolds}, J. Differential Geom. Volume 53, Number 2 (1999), 205--236.

\bibitem{DZ05}  J.-P. Dufour, N.T. Zung, {\em Poisson structures and their normal forms}, Progress in Mathematics, 242. Birkh\"auser Verlag, Basel, 2005.

\bibitem{GSV15} L. Gac\'ia-Naranjo, P. Su\'arez-Serrato, R. Vera, {\em Poisson structures on smooth $4$--manifolds}, Lett. Math. Phys. Vol. 105, no. 11 (2015), 1533--1550.


\bibitem{Ho204} K. Honda, {\em Local properties of self-dual harmonic 2-forms on a 4-manifold},  J. reine angew. Math. 577 (2004), 105--116.

\bibitem{GG}  M. Golubitsky, V. Guillemin, {\it Stable mappings and their singularities}, Graduate Texts in Mathematics, Vol. 14. Springer-Verlag, New York-Heidelberg, 1973.


\bibitem{LGPV13} C. Laurent-Gengoux, A. Pichereau, P. Vanhaecke, {\em Poisson structures,} Grundlehren der Mathematischen Wissenschaften, 347. Springer, Heidelberg, 2013.

\bibitem{L09}Y. Lekili, {\em Wrinkled fibrations on near-symplectic manifolds}, Geometry \& Topology 13 (2009), 277--318.

\bibitem{M71} J. Mather, {\em Stability of $C^{\infty}$ mappings VI: The nice dimensions}, from: "Proceedings
of Liverpool Singularities Symposium I (1969-70)", Lecture Notes in Mathematics 192, Springer, Berlin (1971) 207--253.

\bibitem{P07} T. Perutz, {\em Zero-sets of near-symplectic forms}, Jour.  Symp. Geom.,  vol.4, no.3, (2007) 237--257.


\bibitem{ST16} P. Su\'arez-Serrato, J. Torres Orozco, {\em Poisson structures on wrinkled fibrations}, Bol. Soc. Mat. Mex. (3) 22 (2016), no. 1, 263--280.


\bibitem{V94} I. Vaisman, {\em Lectures on the Geometry of Poisson Manifolds}, Birkh\"auser, Basel, 1994.

\bibitem{V16} R. Vera, {\em Near-symplectic $2n$-manifolds}, Alg. Geom. Topol. 16 no.3 (2016), 1403--1426.

\bibitem{W83} A. Weinstein, {\em The local structure of Poisson manifolds,} J. Differential Geom. 18 (1983), no. 3, 523--557. 

\bibitem{W87} \bysame , {\em Symplectic groupoids and Poisson manifolds,} Bull. Amer. Math. Soc. (N.S.) 16 (1987), no. 1, 101--104.


\bibitem{X91} P. Xu, {\em Morita equivalence of Poisson manifolds,} Comm. Math. Phys. 142 (1991), no. 3, 493--509. 

\bibitem{X92} \bysame , {\em Morita equivalence and symplectic realizations of Poisson manifolds}, Ann. Sci. \'Ecole Norm. Sup. (4) 25 (1992), no. 3, 307--333.

\end{thebibliography}
\end{document}